\documentclass[11pt]{article}%
\usepackage{amsmath}
\usepackage{amsfonts}
\usepackage{amssymb}
\usepackage{graphicx}%
\usepackage[colorlinks,linkcolor=blue,anchorcolor=blue,citecolor=blue]{hyperref}
\allowdisplaybreaks

\newenvironment{keywords}{{\bf Keywords: }}{}

\newtheorem{theorem}{Theorem}[section]

\newtheorem{corollary}[theorem]{Corollary}
\newtheorem{definition}[theorem]{Definition}
\newtheorem{example}[theorem]{Example}
\newtheorem{lemma}[theorem]{Lemma}
\newtheorem{proposition}[theorem]{Proposition}
\newtheorem{remark}[theorem]{Remark}

\numberwithin{equation}{section}

\newenvironment{proof}[1][Proof]{\noindent \textbf{#1.} }{\  \rule{0.5em}{0.5em}}

\begin{document}

\title{Local time and Tanaka formula for  the $G$-Brownian motion\footnote{ Corresponding address:
Institute of Mathematical Economics,
Bielefeld University, Postfach 100131, 33501 Bielefeld, Germany}}

\author{Qian LIN  $^{1,2}$\footnote{     {\it Email address}:
linqian1824@163.com}
\\
{\small $^1$School of Mathematics, Shandong University,  Jinan 250100,   China;}\\
{\small $^2$ Laboratoire de Math\'ematiques, CNRS UMR 6205,
Universit\'{e} de Bretagne Occidentale,}\\ {\small 6, avenue Victor
Le Gorgeu, CS 93837, 29238 Brest cedex 3, France.}  }

\date{}

\maketitle
\begin{abstract} In this paper, we study the notion of  local time and obtain the Tanaka formula for  the $G$-Brownian
motion. Moreover, the joint continuity of the local time of the
$G$-Brownian motion is obtained and its quadratic variation is
proven. As an application, we generalize It\^{o}'s formula  with
respect to the $G$-Brownian motion to convex functions.
\end{abstract}

\noindent

\section{Introduction }
The objective of the present paper is to study the local time as
well  as the Tanaka formula for the $G$-Brownian motion.

   Motivated by  uncertainty  problems, risk  measures  and
superhedging in finance,  Peng has introduced a new notion of
nonlinear expectation, the so-called $G$-expectation (see
\cite{Peng2006},  \cite{Peng:2007}, \cite{Peng2008}, \cite
{Peng:2009}, \cite {Peng:2010}),
   which is  associated with   the following  nonlinear heat equation
 $$\left\{\begin{array}{l l}
      \dfrac{\partial u(t, x)}{\partial t}=G(\dfrac{\partial ^{2}u(t, x)}{\partial x^{2}}),\quad (t, x)\in[0,
\infty)\times\mathbb{R},\\
       u(0, x)=\varphi(x),
         \end{array}
  \right.$$
 where, for given parameters $0\leq\underline{\sigma}\leq \overline{\sigma}$, the  sublinear function  $G$
 is defined as follows:
 $$G(\alpha)=\frac{1}{2}(\overline{\sigma}^{2}
 \alpha^{+}-\underline{\sigma}^{2}\alpha^{-}),\quad \alpha \in \mathbb{R}.$$
The $G$-expectation represents a special case of general nonlinear
expectations $\mathbb{\hat{E}}$ whose importance stems from the fact
that they are related to risk measures $\rho$ in finance by the
relation $\mathbb{\hat{E}}[X]=\rho(-X)$, where $X$ runs the class of
contingent claims. Although the $G$-expectations represent only a
special case, their importance inside the class of nonlinear
expectations stems from the stochastic analysis which it allows to
develop, in particular, such  important results as the law of large
numbers and the central limit theorem under nonlinear expectations,
obtained by Peng \cite{Peng2007}, \cite{Peng2008-1} and
\cite{Peng:2010a}.

Together with the notion of $G$-expectations Peng also introduced
the related $G$-normal distribution and the $G$-Brownian motion. The
$G$-Brownian motion is a stochastic process  with stationary and
independent  increments and its quadratic variation process is,
unlike the classical case of a linear expectation, a non
deterministic process. Moreover,  an It\^{o} calculus for the
$G$-Brownian motion has been developed recently in \cite{Peng2006},
\cite{Peng:2007} and \cite{Peng2008} and \cite{Peng2009}.

The fundamental notion of the local time for the classical Brownian
motion has been introduced by L\'{e}vy in \cite{L1948}, and   its
existence was established by Trotter \cite{TR} in 1958. In virtue of
its various applications in stochastic analysis  the notion of local
time and the related Tanaka formula have been well studied by
several authors, and it has also been extended to other classes of
stochastic processes.

It should be expected that the notion of local time and the very
narrowly related Tanaka formula will have the same importance in the
$G$-stochastic analysis on sublinear expectation spaces, which is
developed by Peng  (\cite{Peng2006},  \cite{Peng:2007},
\cite{Peng2008}, \cite {Peng:2009}, \cite {Peng:2010}). However,
unlike the study of the local time for the classical Brownian
motion, its investigation with respect to the $G$-Brownian motion
meets several difficulties: firstly, in contrast to the classical
Brownian motion the $G$-Brownian motion is not defined on a given
probability space but  on a sublinear expectation space. The
$G$-expectation $\mathbb{\hat{E}}$ can be represented as the upper
expectation of a subset of linear expectations $\{E_P, P\in\mathcal
{P}\} $, $i.e.,$
$$\mathbb{\hat{E}}[\cdot]=\sup\limits_{P\in\mathcal {P}} E_P[\cdot],$$
 where $P$ runs a large class of probability measures $\mathcal
 {P}$, which are  mutually singular. Let us point out that this is a very common
situation in financial models under volatility uncertainty (see
\cite{ALP}, \cite{DM}, \cite{L}).
 Secondly, related with  the novelty of
the theory of $G$-expectations, there is short of a lot of tools in
the $G$-stochastic analysis, among them there is, for instance, the
dominated convergence theorem. In order to point out the
difficulties, in Section 3, ad hoc definition of the local time for
the $G$-Brownian motion will be given independently of the
probability measures $P\in\mathcal {P} $, and the limits  of the
classical stochastic analysis in the study of this local time will
be  indicated.

Our paper is organized as follows: Section 2 introduces the
necessary notations and preliminaries and it gives a short recall of
some elements of the $G$-stochastic analysis which will be used in
what follows. In Sections 3,  an intuitive approach to the local
time shows its limits in the frame of the classical stochastic
analysis and works out the problems to be studied. In Section 4, we
define the notion of the local time for the $G$-Brownian motion and
prove the related Tanaka formula, which non-trivially generalizes
the classical one. Moreover, by using Kolmogorov's continuity
criterion under the nonlinear expectation the existence of a jointly
continuous version of the local time is shown. Finally, Section 5
investigates the quadratic variation of the local time for the
$G$-Brownian motion. As an application of the Local time and the
Tanaka formula for the $G$-Brownian motion obtained in Section 4,
Section 6 gives a generalization of the It\^{o} formula for the
$G$-Brownian motion to convex functions, which generalizes the
corresponding result in \cite {Peng2009}.

\section{Notations and preliminaries}

In this section, we introduce some notations and preliminaries of
the theory of sublinear expectations and the related $G$-stochastic
analysis, which will be needed in what follows. More details of this
section can be found in Peng \cite{Peng2006},  \cite{Peng:2007},
\cite{Peng2008}, \cite {Peng:2009} and  \cite {Peng:2010}.

 Let $\Omega$ be a given nonempty set and
$\mathcal{H}$ a linear space of real valued functions defined on
$\Omega$ such that, $1\in \mathcal{H}$ and $|X|\in \mathcal{H}$, for
all $X\in \mathcal{H}$.

\begin{definition}
{\bf A sublinear expectation} $\mathbb{\hat{E}}$ on $\mathcal {H}$
is a functional $\mathbb{\hat{E}}:\mathcal {H}\mapsto \mathbb{R}$
satisfying the following properties: for all $X, Y \in \mathcal
{H}$, we have
\begin{enumerate}
\item[(i)] {\bf Monotonicity:} If $X\geq Y $, then
$\mathbb{\hat{E}}[X]\geq\mathbb{\hat{E}}[Y]$.
\item[(ii)] {\bf Preservation of constants:}
$\mathbb{\hat{E}}[c]=c$, for all $c\in \mathbb{R}$.
\item[(iii)] {\bf Subadditivity:}
$\mathbb{\hat{E}}[X]-\mathbb{\hat{E}}[Y]\leq\mathbb{\hat{E}}[X-Y]$.
\item[(iv)] {\bf Positive homogeneity:}
$\mathbb{\hat{E}}[\lambda X]=\lambda\mathbb{\hat{E}}[X],$ for all $
\lambda \geq 0$.
\end{enumerate}
The triple $(\Omega, \mathcal {H}, \mathbb{\hat{E}})$ is called a
sublinear expectation space.
\end{definition}

\begin{remark}
$\mathcal{H}$ is considered as the space of random variables on
$\Omega$.
\end{remark}

Let us now consider a space of random variables $\mathcal{H}$ with
the additional property of stability with respect to bounded
Lipschitz functions. More precisely, we suppose, if $X_{i}\in
\mathcal{H}$, $i=1,\cdots,d$,
then%
\[
\varphi(X_{1},\cdots,X_{d})\in \mathcal{H}\text{,\  \ for all
}\varphi \in C_{b,Lip}(\mathbb{R}^{d}),
\]
where $C_{b,Lip}(\mathbb{R}^{d})$ denotes the space of all bounded
 Lipschitz  functions on $\mathbb{R}^{d}$.

\begin{definition} In a sublinear expectation space $(\Omega, \mathcal {H},
\mathbb{\hat{E}})$, a random vector $Y=(Y_{1},\cdots, Y_{n}),
\\ Y_{i}\in \mathcal {H}$, is said to be independent under $
\mathbb{\hat{E}}$ from another random vector $X=(X_{1},\cdots,
X_{m}), X_{i} \in \mathcal {H}$,  if for each test function $\varphi
\in C_{b,lip}(\mathbb{R}^{m+n})$ we have
$$\mathbb{\hat{E}}[\varphi(X, Y)]=\mathbb{\hat{E}}[\mathbb{\hat{E}}[\varphi(x, Y)]_{x=X}].$$
\end{definition}

\begin{definition}
Let $X_{1}$ and $X_{2}$ be two $n$--dimensional random vectors
defined
respectively in the sublinear expectation spaces $(\Omega_{1},\mathcal{H}%
_{1},\mathbb{\hat{E}}_{1})${ and
}$(\Omega_{2},\mathcal{H}_{2},\mathbb{\hat {E}}_{2})$. They are
called identically distributed, denoted by $X_{1}\sim X_{2}$, if
\[
\mathbb{\hat{E}}_{1}[\varphi(X_{1})]=\mathbb{\hat{E}}_{2}[\varphi
(X_{2})],\  \  \  \text{for all}\  \varphi \in
C_{b.Lip}(\mathbb{R}^{n}).
\]
\end{definition}

After the above basic definition we introduce now the central notion
of $G$-normal distribution.
\begin{definition}($G$-normal distribution)
 Let be given two reals $\underline{\sigma}, \overline{\sigma}$ with
$0\leq\underline{\sigma}\leq \overline{\sigma}.$ A random variable
$\xi$ in a sublinear expectation space $(\Omega, \mathcal {H},
\mathbb{\hat{E}})$ is called $G_{\underline{\sigma},
\overline{\sigma}}$-normal distributed,denoted by $\xi\sim \mathcal
{N}(0, [{\underline{\sigma}^{2}, \overline{\sigma}}^{2}])$,  if for
each $\varphi \in C_{b,lip}(\mathbb{R})$, the following function
defined by
$$u(t, x):=\mathbb{\hat{E}}[\varphi(x+\sqrt{t}\xi)], \quad (t,x)\in [0, \infty)\times\mathbb{R},$$
is the unique 、  viscosity solution  of the following parabolic
partial differential equation :
$$\left\{\begin{array}{l l}
      \partial_{t}u(t, x)=G(\partial_{xx}^{2}u(t, x)),\quad (t, x)\in[0,
\infty)\times\mathbb{R},\\
       u(0, x)=\varphi(x).
         \end{array}
  \right.$$
 Here  $G=G_{\underline{\sigma},
\overline{\sigma}}$
 is the following sublinear function parameterized by $\underline{\sigma}$
 and $\overline{\sigma}$:
 $$G(\alpha)=\frac{1}{2}(\overline{\sigma}^{2}\alpha^{+}-\underline{\sigma}^{2}\alpha^{-}), \quad \alpha \in \mathbb{R}$$
 (recall that $\alpha^{+}=\text{max}\{0,\alpha\}$ and
$\alpha^{-}=-\text{min}\{0,\alpha\}$).
\end{definition}

\begin{definition}
 A process $B=\{B_{t},t\geq 0\}\subset \mathcal {H}$ in a sublinear expectation space $(\Omega, \mathcal {H},
 \mathbb{\hat{E}})$ is called a G-Brownian motion if  the following
 properties are satisfied:
\begin{enumerate}
\item[(i)] $B_{0}=0$;
\item[(ii)] for each $t, s \geq0$, the difference $B_{t+s}-B_{t}$ is
$\mathcal {N}(0, \ [\underline{\sigma}^{2}s,
\overline{\sigma}^{2}s])$-distributed and is independent from
$(B_{t_{1}}, \cdots, B_{t_{n}})$, for all $n\in\mathbb{N}$ and
$0\leq t_{1}\leq\cdots \leq t_{n}\leq t$.
\end{enumerate}
\end{definition}

 Throughout this paper, we let from now on $\Omega=C_{0}(\mathbb{R^{+}})$ be the space of all
real valued continuous functions $(\omega_{t})_{t\in
\mathbb{R^{+}}}$ with $\omega_{0}=0$, equipped with the distance
$$\rho(\omega^{1}, \omega^{2})
=\sum\limits_{i=1}^{\infty}2^{-i}\Big[(\max\limits_{t\in[0,i]}|\omega_{t}^{1}-\omega_{t}^{2}|)\wedge1\Big],
\ \omega_{t}^{1}, \omega_{t}^{2}\in\Omega.$$ We denote by
$\mathcal{B}(\Omega)$ the Borel $\sigma$-algebra on $\Omega$. We
also set, for each $t\in \lbrack0,\infty)$, $\Omega_{t}:=\{
\omega_{\cdot \wedge t}:\omega \in \Omega \}$ and $\mathcal
{F}_{t}:=\mathcal{B}(\Omega_{t})$. Moreover, we will work with the
following spaces:

\begin{itemize}
\item $L^{0}(\Omega)$: the space of all $\mathcal{B}(\Omega)$-measurable real valued  functions on $\Omega$;

\item $L^{0}(\Omega_{t})$: the space of all
$\mathcal{B}(\Omega_{t})$-measurable real valued  functions on
$\Omega$;

\item $L_{b}(\Omega)$
: the space of all bounded elements in $L^{0}(\Omega)$;

\item  $L_{b}(\Omega_{t})$: the space of all bounded elements in
 $L^{0}(\Omega_{t})$.
\end{itemize}

In \cite{Peng2006}, a $G$-Brownian motion is constructed on a
sublinear expectation space
$(\Omega,\mathbb{L}_{G}^{p}(\Omega),\mathbb{\hat{E}})$,  where $\mathbb{L}%
_{G}^{p}(\Omega)$  is the Banach space defined as closure of
$\mathcal{H}:=\{\varphi(\omega_{t_{1}},\cdots,\omega_{t_{d}}),
\text{ for all }\varphi \in C_{b,Lip}(\mathbb{R}^{d}), 0\leq
t_{1}< \cdots < t_{d}, d\geq 1\} $ with respect to the norm $\left \Vert X\right \Vert _{p}:=\mathbb{\hat{E}}%
[|X|^{p}]^{1/p}, 1\leq p < \infty$. In this space the coordinate
process $B_{t}(\omega)=\omega_{t}$, $t\in \lbrack0,\infty)$, $\omega
\in\Omega$,  is a $G$-Brownian motion. Let us point out that the
space $C_{b}(\Omega)$ of the bounded continuous functions on
$\Omega$ is a subset of $\mathbb{L} _{G}^{1}(\Omega)$. Moreover,
there exists a weakly compact family $\mathcal{P}$ of probability
measures on $(\Omega ,\mathcal{B}(\Omega))$ such that
\begin{eqnarray}\label{e22}
\mathbb{\hat{E}}[\cdot]=\sup_{P\in \mathcal{P}}E_{P}[\cdot].
\end{eqnarray}
The Choquet capacity:
\[
\hat{c}(A):=\sup_{P\in \mathcal{P}}P(A),\  \ A\in \mathcal{B}(\Omega).
\]

\begin{definition}
A set $A\subset \Omega$ is called polar if $\hat{c}(A)=0$. A
property is said to hold \textquotedblleft
quasi-surely\textquotedblright \ (q.s.) if it holds outside a polar
set.
\end{definition}

\begin{definition}
A mapping $X$ on $\Omega$
 with values in a topological space is said to
be quasi-continuous if for all $\varepsilon>0$, there exists an open
set $O$ with $\hat{c}(O)<\varepsilon$ such that $X|_{O^{c}}$ is
continuous.
\end{definition}

The family of probability measures $\mathcal{P}$
 allows to characterize the space $\mathbb{L}_{G}^{p}(\Omega)$ as follows:
\begin{eqnarray*}
\mathbb{L}_{G}^{p}(\Omega)=\Big\{X\in
L^{0}(\Omega):\lim\limits_{n\rightarrow\infty}\sup\limits_{P\in
\mathcal{P}} E_{P}[|X|^{p}1_{|X|>n}]=0 \text{, and }X\text{ is
}\hat{c}\text{-quasi  continuous}\Big\}.
\end{eqnarray*}
We also introduce the following spaces, for all $p>0$,

\begin{itemize}
\item $\mathcal{L}^{p}:=\Big\{X\in L^{0}(\Omega):\mathbb{\hat{E}}[|X|^{p}%
]=\sup\limits_{P\in \mathcal{P}}E_{P}[|X|^{p}]<\infty \Big\}$;

\item $\mathcal{N}:=\Big\{X\in L^{0}(\Omega):X=0$, $\hat{c}$-quasi surely
(q.s.).$\Big\}$
\end{itemize}
Obviously, $\mathcal{L}^{p}$ and $\mathcal{N}$ are linear spaces and
$\Big\{X\in
L^{0}(\Omega):\mathbb{\hat{E}}[|X|^{p}]=0\Big\}=\mathcal{N}$, for
all $p>0$. We put $\mathbb{L}^{p}:=\mathcal{L}^{p}/\mathcal{N}$. As
usual, we will not  make here the distinction between classes and
their representatives.

The following three propositions  can be consulted in \cite{DHP} and
\cite{HP} .

\begin{proposition}
\label{pr8} For every monotonically decreasing sequence  $\{X_{n}
\}_{n=1}^{\infty}$ of nonnegative functions in  $C_{b}(\Omega)$,
which converges to zero q.s. on $\Omega$, it holds
$\lim\limits_{n\rightarrow \infty}\mathbb{\hat{E}}[X_{n}]=0$.
\end{proposition}

\begin{proposition}
\label{pr9} For all $p>0$, we have

\begin{enumerate}
\item  $\mathbb{L}^{p}$ is a Banach space with respect to the norm
$\left \Vert X\right \Vert _{p}:=\left(
\mathbb{\hat{E}}[|X|^{p}]\right) ^{\frac{1}{p}}$.

\item $\mathbb{L}_{G}^{p}$ is the completion of $C_{b}(\Omega)$ with respect to the norm
$\left \Vert \cdot\right \Vert _{p}$.
\end{enumerate}
\end{proposition}

 We denote by $\mathbb{L}_{\ast}^{p}(\Omega)$ the completion of $L_{b}(\Omega)$ with respect to the norm
$\left \Vert \cdot\right \Vert _{p}$.

\begin{proposition}\label{pr1}
For a given $p\in(0,+\infty]$, let $\{X_{n}\}_{n=1}^{\infty}\subset
\mathbb{L}^{p}$ be a sequence converging to $X$ in $\mathbb{L}^{p}$.
Then there exists a subsequence $(X_{n_{k}})$ which converges to $X$
quasi-surely in the sense that it converges to $X$ outside a polar
set.
\end{proposition}

 We now recall the definition of quadratic variation
process of the $G$-Brownian motion. We use
$\{0=t_{0}<t_{1}\cdots<t_{n}=t\}$ to denote a partition of $[0,t]$
such that
  $\max\{t_{i+1}-t_{i}, 0\leq i\leq n-1\}\rightarrow 0,$ as
  $n\rightarrow\infty.$ Then the quadratic variation
process of the $G$-Brownian motion is defined as follows:
\begin{eqnarray*}
\langle
B\rangle_{t}:=\mathbb{L}_{G}^{2}\lim\limits_{n\rightarrow\infty}\sum\limits_{i=0}^{n-1}[B_{t_{i+1}^{n}\wedge
t}-B_{t_{i}^{n}\wedge t}]^{2} =B^{2}_{t}-2\int^{t}_{0}B_{s}dB_{s}.
\end{eqnarray*}
$\langle B \rangle$ is continuous and increasing outside a polar
set.

In \cite{Peng2009}, a  generalized It\^{o} integral and a
generalized It\^{o} formula with respect to the $G$-Brownian motion
are discussed as follows.

 For arbitrarily fixed $p\geq1$ and $T\in
\mathbb{R}_{+}$, we first consider the following set of step
processes:
\begin{eqnarray}\label{a1}
M_{b,0}(0,T)&=&\Big\{\eta:\eta_{t}(\omega)=\sum_{j=0}^{n-1}\xi_{j}
(\omega)1_{[t_{j},t_{j+1})}(t), 0=t_{0}<\cdots<t_{n}=T, \nonumber \\
& & \quad  \xi_{j}\in L_{b}(\Omega_{t_{j} }), j=0,\cdots,n-1, n\geq
1 \Big\}.
\end{eqnarray}

\begin{definition}
For an $\eta \in M_{b,0}(0,T)$ of the form (\ref{a1}), the related
Bochner integral is
\[
\int_{0}^{T}\eta_{t}(\omega)dt=\sum_{j=0}^{n-1}\xi_{j}(\omega)(t_{j+1}%
-t_{j}).
\]
\end{definition}
For each $\eta \in M_{b,0}(0,T)$, we set
\begin{eqnarray*}
\hat{\mathbb{E}}_{T}[\eta]:=\frac{1}{T}\hat{\mathbb{E}}[\int_{0}^{T}\eta
_{t}dt]=\frac{1}{T}\hat{\mathbb{E}}[\sum_{j=0}^{n-1}\xi_{j}(t_{j+1}-t_{j})],
\end{eqnarray*}
and we  introduce the norm $$||\eta||_{M^{p}_{*}(0,T)}=\Big( \hat{\mathbb{E}%
}[\int_{0}^{T}|\eta_{t}|^{p}dt]\Big)^{1/p}$$ on $ M_{b,0}(0,T)$.
With respect to this norm, $M_{b,0}(0,T)$ can be continuously
extended to a Banach space.

\begin{definition}\label{d1}
For each $p\geq1$, we denote by $M_{\ast}^{p}(0,T)$ the completion of
$M_{b,0}(0,T)$ under the norm
\begin{eqnarray*}
||\eta||_{M^{p}_{*}(0,T)}=\Big(
\hat{\mathbb{E}}[\int_{0}^{T}|\eta_{t}|^{p} dt]\Big)^{1/p}.
\end{eqnarray*}

\end{definition}

\begin{definition}\label{de1}
For every $\eta \in M_{b,0}(0,T)$ of the form (\ref{a1})
\[
\eta_{t}(\omega)=\sum_{j=0}^{n-1}\xi_{j}(\omega)1_{[t_{j},t_{j+1}%
)}(t),
\]
the It\^{o} integral
\[
I(\eta)=\int_{0}^{T}\eta_{s}dB_{s}:=\sum_{j=0}^{n-1}\xi_{j}(B_{t_{j+1}%
}-B_{t_{j}})\mathbf{.}%
\]

\end{definition}

\begin{lemma}\label{le3}
{ { { \label{bdd}The mapping $I: M_{b,0}(0,T)\rightarrow
\mathbb{L}_{\ast} ^{2}(\Omega_{T})$ is a  continuous, linear
mapping. Thus it can be continuously extended to
$I:M_{\ast}^{2}(0,T)\rightarrow \mathbb{L}_{\ast}^{2}(\Omega_{T})$.
Moreover, for all $\eta \in M_{\ast}^{2}(0,T)$, we have
\begin{align}
\mathbb{\hat{E}}[\int_{0}^{T}\eta_{s}dB_{s}]  &  =0,\  \  \\
\mathbb{\hat{E}}[(\int_{0}^{T}\eta_{s}dB_{s})^{2}]  &  \leq \overline{\sigma
}^{2}\hat{\mathbb{E}}[\int_{0}^{T}\eta_{s}^{2}ds]. %
\end{align}
} } }
\end{lemma}

\begin{definition}
\label{Def4.4}For each fixed $p>0$, we denote by $\eta \in
M_{\omega}^{p}(0,T)$ the space of  stochastic processes
$\eta=(\eta_{t})_{t\in [0,T]}$ for which there exists a sequence of
increasing stopping times $\{ \sigma_{m}\}_{m=1}^{\infty}$ (which
can depend on the process $\eta$), with $\sigma_{m}\uparrow T$,
quasi-surely, such that $\eta 1_{[0,\sigma_{m}]}\in
M_{\ast}^{p}(0,T)$ and
\[
\widehat{c}(\int_{0}^{T}|\eta_{s}|^{p}ds<\infty)=1.
\]
\end{definition}

The generalized It\^{o} formula is obtained in \cite{Peng2009}.
\begin{theorem}\label{th1}
Let $\varphi \in C^{2}(\mathbb{R})$ and
$$X_{t}=X_{0}+\int_{0}^{t}\alpha
_{s}ds+\int_{0}^{t}\eta_{s}d\langle
B\rangle_{s}+\int_{0}^{t}\beta_{s}dB_{s}, \ \text{for all} \
t\in[0,T],$$
where $\alpha,\eta$ in $M_{\omega}^{1}(0,T)$ and $\beta \in M_{\omega}%
^{2}(0,T)$. Then, for each $0\leq t\leq T$, we have
\begin{align*}
\varphi(X_{t})-\varphi(X_{0})   =& \int_{0}^{t}\partial_{x}\varphi(X_{u}%
)\beta_{u}dB_{u}+\int_{0}^{t}\partial_{x}\varphi(X_{u})\alpha_{u}du\\
&  +\int_{0}^{t}[\partial_{x}\varphi(X_{u})\eta_{u}+\frac{1}{2}\partial
_{xx}^{2}\varphi(X_{u})\beta_{u}^{2}]d\langle B\rangle_{u}.
\end{align*}
\end{theorem}

\begin{example}
For all $\varphi \in C^{2}(\mathbb{R})$ and $t\geq 0$,   we have
\begin{eqnarray*}
\varphi(B_{t})=\varphi(0)+\int_{0}^{t}\varphi_{x}(B_{s}
)dB_{s}+\frac{1}{2}\int_{0}^{t}\varphi_{xx}(B_{s})d\left \langle
B\right \rangle _{s}.
\end{eqnarray*}
\end{example}

\section{An intuitive approach to the local time}
We consider the  $G$-Brownian motion $B$.  Recall that $B$ is
continuous and a $P$-martingale, for all $P\in\mathcal {P}$ (see
(\ref{e22})), and its quadratic variation process $\langle B
\rangle$ is continuous and increasing outside a polar set $\mathcal
{N}$. Let us give a definition of the local time of $B$, which is
independent of the underlying probability measure $P\in\mathcal
{P}$. For all $a\in\mathbb{R}$, $t\in [0,T] $, and $\omega\in\Omega
\setminus \mathcal {N}$,
\begin{eqnarray*}
L^{a}_{t}(\omega)=\left\{
\begin{array}{rcl}
\overline{\lim\limits_{\varepsilon\downarrow0}}\dfrac{1}{2\varepsilon}\int_{0}^{t}1_{
(a-\varepsilon, a+\varepsilon)}(B_{s}(\omega))d\langle B
\rangle_{s}(\omega), &&
\overline{\lim\limits_{\varepsilon\downarrow0}}\dfrac{1}{2\varepsilon}\int_{0}^{t}1_{
(a-\varepsilon, a+\varepsilon)}(B_{s}(\omega))d\langle B \rangle_{s}(\omega)<\infty;\\
0, && \text{otherwise}.
\end{array}
\right.
\end{eqnarray*}
We now study $L^{a}$ under each probability $P\in\mathcal {P}$. For
this we consider the stochastic integral
\begin{eqnarray}\label{e1}
M_{t}^{P}=\int_{0}^{t}sgn(B_{s}-a)dB_{s}, \ t\in[0,T],
\end{eqnarray}
 under $P\in\mathcal {P}$. Indeed, under $P$ the $G$-Brownian motion is a continuous square integrable
 martingale, so that the above stochastic integral under $P$ is well-defined. We emphasize that the process
 $M^{P}$ is defined only $P-a.s.$ Let $L^{a,P}$ be the
local time associated with $B$ under $P$, defined by the relation:
\begin{eqnarray*}
L^{a, P}_{t}=|B_{t}-a|-|a|-M_{t}^{P},  \ t\in[0,T].
\end{eqnarray*}
It is well known that $L^{a, P}_{t}$ admits a $P$-modification that
 is continuous in $(t,a)$, and
\begin{eqnarray*}
L^{a,P}_{t}=\lim\limits_{\varepsilon\downarrow0}\dfrac{1}{2\varepsilon}\int_{0}^{t}1_{
(a-\varepsilon, a+\varepsilon)}(B_{s})d\langle B \rangle_{s},
P-a.s.,  \ t\in[0,T],
\end{eqnarray*}
Thus, due to the definition of  $L^{a}$, we have
$$L^{a}_{t}=L^{a,P}_{t}, P-a.s, t\in[0,T].$$ Consequently, $L^{a}$
 has a continuous $P$-modification, for all $P\in\mathcal {P}$.
However, $P$ is not dominated by a probability measure. Thus the
question, if we can find a continuous modification of $L^{a}$ or
even a jointly continuous modification of $(\alpha, t)\mapsto
 L_t^{\alpha}$ under $\mathbb{\hat{E}}$, i.e., with respect to  all $P\in\mathcal
 {P}$, is a nontrivial one which cannot be solved in the frame of the classical stochastic analysis.
  The other question which has its own importance is that of the integrability
 of $sgn(B_{\cdot}-a)$ in the framework of $G$-expectations.
Indeed, in (\ref{e1}) we have considered the stochastic integral of
$sgn(B_{\cdot}-a)$ under $P$, for each $P\in\mathcal {P}$
separately. However, the main difficulty is that the above family of
probability measures is not dominated by one of these  probability
measures. To overcome this difficulty,  we shall show that
$sgn(B_{\cdot}-a)$ belongs to a suitable space of processes
integrable  with respect  to the $G$-Brownian motion $B$. It should
be expected that this space is $M_{\ast}^{2}(0,T)$ (see Definition
\ref{d1}). This turns out to be correct if $\underline{\sigma}>0$,
but it is not clear at all  in the case of $\underline{\sigma}=0$.
For this reason a larger space $\widetilde{M}_{\ast}^{2}(0,T)$ (see
Definition \ref{d2}) is introduced in Section 4, and its
relationship with $M_{\ast}^{2}(0,T)$ is discussed.
 To get a jointly continuous modification of $(\alpha,
t)\mapsto L_t^{\alpha}$  under $\mathbb{\hat{E}}$ in the framework
of $G$-expectations, we use, in addition to the integrality of
$sgn(B_{\cdot}-a)$ with respect to the $G$-Brownian motion $B$, an
approximation approach which is different from the classical method.

\section{Local time and Tanaka formula  for the $G$-Brownian motion}
The objective of this section is to study the notion of local time
and to obtain Tanaka formula for the $G$-Brownian motion. First, we
shall   generalize It\^{o} integral with respect to the $G$-Brownian
motion, which plays an important role in what follows.

For each $\eta \in M_{b,0}(0,T)$, we  introduce the norm $$||\eta||_{\widetilde{M}^{p}_{*}(0,T)}=\Big( \hat{\mathbb{E}%
}[\int_{0}^{T}|\eta_{t}|^{p}d\langle B \rangle_{t}]\Big)^{1/p}$$ on
$ M_{b,0}(0,T)$. With respect to this norm, $M_{b,0}(0,T)$ can be
continuously extended to a Banach space.

\begin{definition}\label{d2}
For each $p\geq1$, we denote by $\widetilde{M}_{\ast}^{p}(0,T)$ the
completion of $M_{b,0}(0,T)$ under the norm
\[
||\eta||_{\widetilde{M}^{p}_{*}(0,T)}=\Big( \hat{\mathbb{E}}[\int_{0}^{T}|\eta_{t}|^{p}%
d\langle B \rangle_{t}]\Big)^{1/p}.
\]
\end{definition}

\begin{remark}
The elements of $\widetilde{M}_{\ast}^{p}(0,T)$ are defined only
$d\langle B \rangle_{t}dP$-a.e., for all $P\in\mathcal {P}$. While
for $\underline{\sigma}>0$, this means that
$\eta\in\widetilde{M}_{\ast}^{p}(0,T)$ is defined $dtdP$-a.e., for
all $P\in\mathcal {P}$ (and, may, hence, not be well defined with
respect to $P\in\mathcal {P}$ on $\Gamma\times\Omega$, for
$\Gamma\subset[0,T]$ of Lebesgue measure zero). In the case
$\underline{\sigma}=0$, this has the consequence that, if for some
$P\in\mathcal {P}, 0\leq t< t+\varepsilon\leq T$, $P\{\langle B
\rangle_{t+\varepsilon}-\langle B \rangle_{t}=0\}>0$, the elements
of $\widetilde{M}_{\ast}^{p}(0,T)$ are  not well defined  on
$[t,t+\varepsilon]\times \Omega$ under $P$.

\end{remark}

\begin{remark}\label{re1}
For every $p\geq 1$, it is easy to check that
$M_{\ast}^{p}(0,T)\subset\widetilde{M}_{\ast}^{p}(0,T)$. Moreover,
if $\underline{\sigma}>0$, then
$M_{\ast}^{p}(0,T)=\widetilde{M}_{\ast}^{p}(0,T).$   If
$\underline{\sigma}=0$, then for every $p\geq 1$,
$M_{\ast}^{p}(0,T)$ is a strict subset of
$\widetilde{M}_{\ast}^{p}(0,T)$. In fact, due to Denis, Hu and Peng
\cite{DHP}
\begin{eqnarray*}
\hat{\mathbb{E}}[\int_{0}^{T}\frac{1}{s}1_{\{B_{s}=0\}}d\langle B
\rangle_{ s}] &=& \sup\limits_{u\in \mathcal
{A}}E_{P_{0}}[\int_{0}^{T}\frac{1}{s}1_{\{\int_{0}^{s} u_{s}d
B_{s}=0\}}u_{s}^{2}ds],
\end{eqnarray*}
where $P_{0}$ is the Wiener measure on
$\Omega=C_{0}(\mathbb{R}_{+})$,

$\mathcal {F}_{t}:=\sigma\{B_{s},  0\leq s\leq t\}\vee\mathcal {N}$,
$\mathcal {N}$ is the collection of $P_{0}$-null sets, $\mathcal
{F}=\{\mathcal {F}_{t}\}_{t\geq 0}$,

 $\mathcal {A}:=\{u: u \ \text{is
a}\ \mathcal {F}_{t}-\text{adapted process such that}\ 0\leq u\leq
\overline{\sigma}\}$. \\ Given $u\in \mathcal {A}$, let us put
$M_{s}=\int_{0}^{s} u_{s}d B_{s}, s\in [0,T].$ Then $M$ is a
continuous square integrable martingale. Thanks to the time
transformation for continuous martingales (see  Theorem V.1.6 in
Revuz and Yor \cite{RY1999}), there exists a Brownian motion $N$ on
$(\Omega,\mathcal {F},P_{0})$ (or, eventually, an enlargement of
this probability space) such that $M_{s}=N_{<M>_{s}}, s\in [0,T].$
It follows that
\begin{eqnarray*}
&&E_{P_{0}}[\int_{0}^{T}1_{\{M_{s}=0\}}d <M>_{s}] =
E_{P_{0}}[\int_{0}^{T}1_{\{N_{<M>_{s}}=0\}}d <M>_{s}]\\
&=&E_{P_{0}}[\int_{0}^{<M>_{T}}1_{\{N_{s}=0\}}ds]
\leq E_{P_{0}}[\int_{0}^{\infty}1_{\{N_{s}=0\}}ds]\\
&=& \int_{0}^{\infty} P_{0}\{N_{s}=0\} ds=0.
\end{eqnarray*}
Consequently,
\begin{eqnarray*}
E_{P_{0}}[\int_{0}^{T}1_{\{M_{s}=0\}}d <M>_{s}] =0,
\end{eqnarray*}
and, hence,
\begin{eqnarray*}
0=E_{P_{0}}[\int_{0}^{T} \frac{1}{s}1_{\{M_{s}=0\}}d <M>_{s}]
=E_{P_{0}}[\int_{0}^{T}\frac{1}{s}1_{\{\int_{0}^{s} u_{s}d
B_{s}=0\}}u_{s}^{2}ds].
\end{eqnarray*}
This implies
\begin{eqnarray*}
\hat{\mathbb{E}}[\int_{0}^{T}\frac{1}{s}1_{\{B_{s}=0\}}d\langle B
\rangle_{ s}] =  \sup\limits_{u\in \mathcal
{A}}E_{P_{0}}[\int_{0}^{T}\frac{1}{s}1_{\{\int_{0}^{s} u_{s}d
B_{s}=0\}}u_{s}^{2}ds]=0,
\end{eqnarray*}
$i.e., \big\{s^{-\frac{1}{p}}1_{\{B_{s}=0\}}\big\}_{s\in[0,T]}$ can
be identified with the process identically equal to $0$ in
$\widetilde{M}_{\ast}^{p}(0,T)$, and so it belongs to this space.
 On the other hand, for $u_{s}=0,
s\in [0,T],$ as the optimal control,
\begin{eqnarray*}
\hat{\mathbb{E}}[\int_{0}^{T}\frac{1}{s}1_{\{B_{s}=0\}}ds] &=&
\sup\limits_{u\in \mathcal
{A}}E_{P_{0}}[\int_{0}^{T}\frac{1}{s}1_{\{\int_{0}^{s} u_{s}d
B_{s}=0\}}ds]\\
&=& \int_{0}^{T}\frac{1}{s}ds=+\infty.
\end{eqnarray*}
Therefore, $\big\{s^{-\frac{1}{p}}1_{\{B_{s}=0\}}\big\}_{s\in[0,T]}$
cannot belong to $M_{\ast}^{p}(0,T)$.
\end{remark}

For every $\eta \in M_{b,0}(0,T)$, the It\^{o} integral with respect
to the $G$-Brownian motion is defined in Definition \ref{de1}.
\begin{lemma}\label{le4}
{ { { \label{bdd}The mapping $I: M_{b,0}(0,T)\rightarrow \mathbb{L}
_{\ast}^{2}(\Omega_{T})$ is a  continuous and linear mapping. Thus
it can be continuously extended to
$I:\widetilde{M}_{\ast}^{2}(0,T)\rightarrow
\mathbb{L}_{\ast}^{2}(\Omega_{T})$. Moreover, for all $\eta \in
\widetilde{M}_{\ast}^{2}(0,T)$, we have
\begin{align*}
\mathbb{\hat{E}}[\int_{0}^{T}\eta_{s}dB_{s}]  &  =0,\  \  \\
\mathbb{\hat{E}}[(\int_{0}^{T}\eta_{s}dB_{s})^{2}]  &  \leq
\hat{\mathbb{E}}[\int_{0}^{T}\eta_{s}^{2}d\langle B \rangle_{s}]. %
\end{align*}
} } }
\end{lemma}

Now we  establish the Burkh\"{o}lder-Davis-Gundy inequality for the
framework of the $G$-stochastic analysis, which will be needed in
what follows.
\begin{lemma}\label{le1}
\label{p3} For each $p>0$, there exists a  constant $c_{p}>0$ such
that \\(i) for all $\eta \in M_{\ast}^{2}(0,T)$,
\[
\hat{\mathbb{E}}[\sup_{0\leq t\leq T}|\int_{0}^{t}\eta_{s}dB_{s}|^{2p}%
]\leq c_{p}\hat{\mathbb{E}}[(\int_{0}^{T}\eta_{s}^{2}d\langle B
\rangle_{ s})^{p}]
\leq\overline{\sigma}^{2p}c_{p}\hat{\mathbb{E}}[(\int_{0}^{T}\eta_{s}^{2}ds)^{p}],
\]
\[
\hat{\mathbb{E}}[\sup_{0\leq t\leq
T}|\int_{0}^{t}\eta_{s}dB_{s}|^{2p}] \geq
\frac{1}{c_{p}}\hat{\mathbb{E}}[(\int_{0}^{T}\eta_{s}^{2}d\langle B
\rangle_{ s})^{p}]
\geq\frac{\underline{\sigma}^{2p}}{c_{p}}\hat{\mathbb{E}}[(\int_{0}^{T}\eta_{s}^{2}ds)^{p}];
\]
(ii) for all $\eta \in \widetilde{M}_{\ast}^{2}(0,T)$,
\[
\frac{1}{c_{p}}\hat{\mathbb{E}}[(\int_{0}^{T}\eta_{s}^{2}d\langle B
\rangle_{ s})^{p}]\leq \hat{\mathbb{E}}[\sup_{0\leq t\leq T}|\int_{0}^{t}\eta_{s}dB_{s}|^{2p}%
]\leq c_{p}\hat{\mathbb{E}}[(\int_{0}^{T}\eta_{s}^{2}d\langle B
\rangle_{ s})^{p}].
\]
\end{lemma}

\begin{proof} We only give the proof of (i), i.e., for $\eta \in
M_{\ast}^{2}(0,T)$. The proof of (ii) is similar.
Since for each $\alpha \in L_{b}(\Omega_{t})$ we have%
\[
\mathbb{\hat{E}}[\alpha \int_{t}^{T}\eta_{s}dB_{s}]=0,
\]
then the process $\int_{0}^{\cdot}\eta _{s}dB_{s}$ is a
$P$-martingale, for all $P\in \mathcal{P}$. Thus from the classical
Burkh\"{o}lder-Davis-Gundy inequality and the relation
$$\underline{\sigma}^{2}t\leq \langle B\rangle _{t} \leq
\bar{\sigma}^{2}t,\ q.s,$$  we have, for all $P\in \mathcal{P}$,
\[
E_{P}[\sup_{0\leq t\leq T}|\int_{0}^{t}\eta_{s}dB_{s}|^{2p}] \leq
c_{p}E_{P}[(\int_{0}^{T}\eta_{s}^{2}d\langle B \rangle_{ s})^{p}]
\leq\overline{\sigma}^{2p}c_{p}E_{P}[(\int_{0}^{T}\eta_{s}^{2}ds)^{p}],
\]
and
\[
E_{P}[\sup_{0\leq t\leq T}|\int_{0}^{t}\eta_{s}dB_{s}|^{2p}] \geq
\frac{1}{c_{p}}E_{P}[(\int_{0}^{T}\eta_{s}^{2}d\langle B \rangle_{
s})^{p}]
\geq\frac{\underline{\sigma}^{2p}}{c_{p}}E_{P}[(\int_{0}^{T}\eta_{s}^{2}ds)^{p}.]
\]
We emphasize that the constant $c_{p}$ coming from the classical
Burkh\"{o}lder-Davis-Gundy inequality, only depends on $p$ but not
on the underlaying probability measure $P$. Consequently, by taking
the supremum over all $P\in \mathcal{P}$ we have
\[
\hat{\mathbb{E}}[\sup_{0\leq t\leq T}|\int_{0}^{t}\eta_{s}dB_{s}|^{2p}%
]\leq c_{p}\hat{\mathbb{E}}[(\int_{0}^{T}\eta_{s}^{2}d\langle B
\rangle_{ s})^{p}]
\leq\overline{\sigma}^{2p}c_{p}\hat{\mathbb{E}}[(\int_{0}^{T}\eta_{s}^{2}ds)^{p}],
\]
and
\[
\hat{\mathbb{E}}[\sup_{0\leq t\leq
T}|\int_{0}^{t}\eta_{s}dB_{s}|^{2p}] \geq
\frac{1}{c_{p}}\hat{\mathbb{E}}[(\int_{0}^{T}\eta_{s}^{2}d\langle B
\rangle_{ s})^{p}]
\geq\frac{\underline{\sigma}^{2p}}{c_{p}}\hat{\mathbb{E}}[(\int_{0}^{T}\eta_{s}^{2}ds)^{p}].
\]
The proof is complete.
\end{proof}\vskip2mm

The following proposition is very important for our approach.
\begin{proposition}\label{pr}
For any real  $a$,   all $\delta> 0 $ and $t\geq 0 $, we have
\begin{eqnarray*}
\hat{\mathbb{E}}[\int_{0}^{t}1_{[a,a+\delta]}(B_{s})d\langle B
\rangle_{ s}] \leq C\delta.
\end{eqnarray*}
Moreover, if $\underline{\sigma}>0$, then we also have
\begin{eqnarray*}
\hat{\mathbb{E}}[\int_{0}^{t}1_{[a,a+\delta]}(B_{s})ds] \leq
C\delta.
\end{eqnarray*}
 Here $C$ is a constant which depends on $t$ but not on  $\delta$ neither on $a$.
\end{proposition}

\begin{proof}
 For  $\delta>0$, we define the $C^{2}-$ function
 $\varphi:\mathbb{R}\rightarrow\mathbb{R}$ such that $\varphi(0)=0,\ |\varphi'(x)|\leq \delta^{-1}
 $, $x\in (\infty, a-\delta]$, and
\[
\varphi''(x)=\left\{
\begin{array}{lll}
0; & if\ x \leq a-\delta,\\
\dfrac{x-a+\delta}{\delta^{3}}; & if\ a-\delta < x \leq a,\\
\dfrac{1}{\delta^{2}}; & if\ a< x \leq a+\delta,\\
-\dfrac{x-a-2\delta}{\delta^{3}}; & if\ a+\delta < x \leq a+2\delta,\\
0; & if\ x \geq a+2\delta.\\
\end{array}
\right.
\]
Then we have $|\varphi'(x)|\leq 4\delta^{-1}$ and $\varphi''(x)\leq
\delta^{-2}$.

By applying the generalized It\^o formula for the $G$-Brownian
motion (see Theorem \ref{th1}) to $\varphi(B_{t})$, we deduce that

\begin{eqnarray*}
\varphi(B_{t})=\int_{0}^{t}\varphi'(B_{s}%
)dB_{s}+\frac{1}{2}\int_{0}^{t}\varphi''(B_{s})d\left \langle
B\right \rangle _{s}.
\end{eqnarray*}
Therefore,
\begin{eqnarray*}\label{}
\int_{0}^{t}1_{[a,a+\delta]}(B_{s})d\langle B \rangle_{ s}
&\leq& \delta^{2}\int_{0}^{t}\varphi''(B_{s})d\langle B \rangle_{ s}\nonumber\\
&=&2\delta^{2}\varphi(B_{t})-2\delta^{2}\int_{0}^{t}\varphi'(B_{s}
)dB_{s}\nonumber\\
&\leq& 8\delta|B_{t}|-2\delta^{2}\int_{0}^{t}\varphi'(B_{s} )dB_{s}.
\end{eqnarray*}
From the above inequalities and Lemma \ref{le3} it follows that
\begin{eqnarray*}\label{}
\mathbb{\hat{E}}[\int_{0}^{t}1_{[a,a+\delta]}(B_{s})d\langle B
\rangle_{ s}] &\leq&
\mathbb{\hat{E}}[8\delta|B_{t}|-2\delta^{2}\int_{0}^{t}\varphi'(B_{s})dB_{s}]\nonumber\\
&\leq&
8\delta\mathbb{\hat{E}}[|B_{t}|]+2\delta^{2}\mathbb{\hat{E}}[-\int_{0}^{t}\varphi'(B_{s})dB_{s}]\nonumber\\
&=& 8\delta\mathbb{\hat{E}}[|B_{t}|]= C\delta.
\end{eqnarray*}
From Lemma \ref{le1} we have
\begin{eqnarray*}
\hat{\mathbb{E}}[\int_{0}^{t}1_{[a,a+\delta]}(B_{s})ds] \leq
\dfrac{1}{\underline{\sigma}^{2}}\hat{\mathbb{E}}[\int_{0}^{t}1_{[a,a+\delta]}(B_{s})d\langle
B \rangle_{ s}]\leq C\delta.
\end{eqnarray*}
 The proof is complete.
\end{proof}\vskip3mm

From the above proposition and Lemma \ref{le1}, we can derive
interesting results as follows:
\begin{corollary}\label{co1}
For any real number $a$ and $t\geq 0 $, we have
\begin{eqnarray*}
\int_{0}^{t}1_{\{ a \}}(B_{s})d\langle B \rangle_{s}=0,\ q.s.
\end{eqnarray*}
\end{corollary}
Moreover, if $\underline{\sigma}>0$, then we have
\begin{eqnarray*}
\int_{0}^{t}1_{\{ a \}}(B_{s})ds=0,\ q.s.
\end{eqnarray*}

The following lemma will play an important role in what follows and
its proof will be given after Theorem \ref{th3}.
\begin{lemma}\label{le20}
 For each $a\in\mathbb{R}$, the process
 $sgn(B_{\cdot}-a)\in\widetilde{M}_{\ast}^{2}(0,T)$.
\end{lemma}

Now we can state Tanaka formula for the $G$-Brownian motion as
follows. For this we denote
\[
sgn(x)=\left\{
\begin{array}{lll}
1; & x> 0,\\
0; & x=0,\\
-1; &  x< 0.
\end{array}
\right.
\]
\begin{theorem}\label{th3}
For any real number $a$ and all $t\geq 0 $, we have
\begin{eqnarray*}
|B_{t}-a|=|a|+\int_{0}^{t}sgn(B_{s}-a)dB_{s}+L^{a}_{t},
\end{eqnarray*}
where
\begin{eqnarray*}
L^{a}_{t}=\lim\limits_{\varepsilon\rightarrow0}\frac{1}{2\varepsilon}\int_{0}^{t}1_{
(a-\varepsilon, a+\varepsilon)}(B_{s})d\langle B \rangle_{s}, \
(\lim \text{in} \ \mathbb{L}^{2}),
\end{eqnarray*}
and $L^{a}$ is an increasing process.\\
 $L^{a}$ is called the local time for G-Brownian
motion at $a$.
\end{theorem}

\begin{proof}
Without loss of generality, we assume that $a=0$. Let us define
$\eta\in C^{\infty}(\mathbb{R})$ by putting
\[
\eta(x)=\left\{
\begin{array}{lll}
C\exp\Big(\dfrac{1}{|x|^{2}-1}\Big); & if\ |x|<1,\\
0 ; &  if\ |x|\geq1,
\end{array}
\right.
\]
where $C$ is the  positive constant satisfying
$\int_{\mathbb{R}}\eta(x)dx=1$.

 For every $n\in \mathbb{N}$, we put
$$\eta_{n}(x):=n\eta(nx), x\in \mathbb{R}.$$ Then $\eta_{n}\in
C^{\infty}(\mathbb{R})$ and $\int_{\mathbb{R}}\eta_{n}(x)dx=1.$

 For any $\varepsilon>0$, we set
\[
\varphi_{\varepsilon}(x)=\left\{
\begin{array}{lll}
\dfrac{1}{2}\big(\varepsilon+\dfrac{x^{2}}{\varepsilon}); & if\ |x|<\varepsilon,\\
|x| ; &  if\ |x|\geq \varepsilon.
\end{array}
\right.
\]
Then, obviously
\[
\varphi'_{\varepsilon}(x)=\left\{
\begin{array}{lll}
\dfrac{x}{\varepsilon}; & if\ |x|\leq\varepsilon,\\
1 ; &  if\ x> \varepsilon, \\
 -1 ; &  if\ x< -\varepsilon,
\end{array}
\right.
\]
and
\[
\varphi''_{\varepsilon}(x)=\left\{
\begin{array}{lll}
\dfrac{1}{\varepsilon}; & if\ |x|<\varepsilon,\\
0 ; &  if\ |x|> \varepsilon
\end{array}
\right.
\]
($\varphi''$ is not defined  at $-\varepsilon$ and $\varepsilon$).

Let us still introduce
$\varphi_{n}:=\varphi_{\varepsilon}*\eta_{n}$, i.e.,
$$\varphi_{n}(x)=\int_{\mathbb{R}}\eta_{n}(x-y)\varphi_{\varepsilon}(y)dy, x\in\mathbb{R}.$$
Then $\varphi_{n}(x)\in C^{\infty}(\mathbb{R})$, $0\leq
\varphi''_{n}\leq \dfrac{1}{\varepsilon}$, $\varphi_{n}\rightarrow
\varphi_{\varepsilon}, \varphi'_{n}\rightarrow
\varphi'_{\varepsilon}$ uniformly in $\mathbb{R}$, and
$\varphi''_{n}\rightarrow \varphi''_{\varepsilon}$ pointwise (except
at $\varepsilon$ and $ - \varepsilon$), as $n\rightarrow \infty$.

By applying the generalized It\^o formula for the $G$-Brownian
motion (see Theorem \ref{th1}) to $\varphi_{n}(B_{t})$, we deduce
that
\begin{eqnarray}\label{equation}
\varphi_{n}(B_{t})=\varphi_{n}(0)+\int_{0}^{t}\varphi_{n}'(B_{s}%
)dB_{s}+\frac{1}{2}\int_{0}^{t}\varphi_{n}''(B_{s})d\left \langle
B\right \rangle _{s}.
\end{eqnarray}
From $\varphi_{n}\rightarrow \varphi_{\varepsilon},
\varphi'_{n}\rightarrow \varphi'_{\varepsilon}$, uniformly in
$\mathbb{R}$, it follows that $$\varphi_{n}(B_{t})\rightarrow
\varphi_{\varepsilon}(B_{t}),\ \varphi_{n}(0)\rightarrow
\varphi_{\varepsilon}(0), $$ and
 $$\int_{0}^{t}\varphi_{n}'(B_{s}
)dB_{s}\rightarrow \int_{0}^{t}\varphi_{\varepsilon}'(B_{s}%
)dB_{s}$$ in $\mathbb{L}^{2}$, as $n\rightarrow \infty$.

Setting $$A_{n,\varepsilon}:=(-\varepsilon-\frac{1}{n},
-\varepsilon+\frac{1}{n})\bigcup (\varepsilon-\frac{1}{n},
\varepsilon+\frac{1}{n}),$$ we observe that $\varphi''_{n}=
\varphi''_{\varepsilon}$ on $A_{n,\varepsilon}^{c}$.

From Corollary \ref{co1} we know that
\begin{eqnarray*}
\int_{0}^{t}1_{\{-\varepsilon, \varepsilon \}}(B_{s})d\langle B
\rangle_{ s}=\int_{0}^{t}1_{\{-\varepsilon\}}(B_{s})d\langle B
\rangle_{ s}+\int_{0}^{t}1_{\{ \varepsilon \}}(B_{s})d\langle B
\rangle_{ s}=0,\ q.s.
\end{eqnarray*}
Let us put $\varphi''_{\varepsilon}(x)=0, x=\pm \varepsilon.$

Since $0\leq \varphi''_{n}\leq \dfrac{1}{\varepsilon}$, we have
$|\varphi''_{n}-\varphi''_{\varepsilon}|\leq\dfrac{2}{\varepsilon}$
on $A_{n,\varepsilon}$. Therefore, we have
\begin{eqnarray*}
&&\mathbb{\hat{E}}[|\int_{0}^{t}\varphi_{n}''(B_{s})d\langle B \rangle_{s}- \int_{0}^{t}\varphi_{\varepsilon}''(B_{s}%
)d\langle B \rangle_{s}|]\\
&\leq&\mathbb{\hat{E}}[\int_{0}^{t}|\varphi_{n}''(B_{s})-
\varphi_{\varepsilon}''(B_{s})|d\langle B \rangle_{s}]\\
&\leq&\dfrac{2}{\varepsilon}\mathbb{\hat{E}}[\int_{0}^{t}1_{A_{n,\varepsilon}}(B_{s})d\langle
B \rangle_{s}]\\
&\leq&\dfrac{2}{\varepsilon}\mathbb{\hat{E}}[\int_{0}^{t}1_{(-\varepsilon-\frac{1}{n},
-\varepsilon+\frac{1}{n})}(B_{s})d\langle
B\rangle_{s}]+\dfrac{2}{\varepsilon}\mathbb{\hat{E}}[\int_{0}^{t}1_{(\varepsilon-\frac{1}{n},
\varepsilon+\frac{1}{n})}(B_{s})d\langle B \rangle_{s}].
\end{eqnarray*}
From Proposition \ref{pr} we conclude that
\begin{eqnarray*}
&&\mathbb{\hat{E}}[|\int_{0}^{t}\varphi_{n}''(B_{s})d\langle B
\rangle_{s}- \int_{0}^{t}\varphi_{\varepsilon}''(B_{s})d\langle B
\rangle_{s}|]\leq \frac{C}{\varepsilon n}\rightarrow 0,\ \text{as}\
n\rightarrow \infty.
\end{eqnarray*}
Therefore, letting $n\rightarrow\infty$ in (\ref{equation}), we
conclude that
\begin{eqnarray}\label{equation1}
\varphi_{\varepsilon}(B_{t})=\dfrac{1}{2}\varepsilon+\int_{0}^{t}\varphi_{\varepsilon}'(B_{s}%
)dB_{s}+\frac{1}{2\varepsilon}\int_{0}^{t}1_{ (-\varepsilon,
\varepsilon)}(B_{s})d\langle B \rangle_{s}.
\end{eqnarray}
From the definition of $\varphi_{\varepsilon}$ it follows that
\begin{eqnarray*}
&&\mathbb{\hat{E}}[(\varphi_{\varepsilon}(B_{t})- |B_{t}|)^{2}]\\
&\leq& \mathbb{\hat{E}}[(\varphi_{\varepsilon}(B_{t})-
|B_{t}|)^{2}1_{|B_{t}|\geq\varepsilon}]
+\mathbb{\hat{E}}[(\varphi_{\varepsilon}(B_{t})- |B_{t}|)^{2}1_{|B_{t}|<\varepsilon}]\\
&=&\mathbb{\hat{E}}[(\varphi_{\varepsilon}(B_{t})-
|B_{t}|)^{2}1_{|B_{t}|<\varepsilon}]\leq \varepsilon^{2}\rightarrow
0, \ \text{as}\  \varepsilon\rightarrow 0.
\end{eqnarray*}
By virtue of Lemma \ref{le1}, we know that
\begin{eqnarray*}
&&\mathbb{\hat{E}}[|\int_{0}^{t}\varphi_{\varepsilon}'(B_{s}%
)dB_{s}-\int_{0}^{t}sgn(B_{s})dB_{s}|^{2}]\\
&\leq& C\mathbb{\hat{E}}[\int_{0}^{t}(\varphi_{\varepsilon}'(B_{s}
)-sgn(B_{s}))^{2}d\langle B \rangle_{s}]\\
&\leq&C\mathbb{\hat{E}}[\int_{0}^{t}1_{ (-\varepsilon,
\varepsilon)}(B_{s})d\langle B \rangle_{s}],
\end{eqnarray*}
and Proposition \ref{pr} allows to conclude that
\begin{eqnarray*}
\mathbb{\hat{E}}[|\int_{0}^{t}\varphi_{\varepsilon}'(B_{s}%
)dB_{s}-\int_{0}^{t}sgn(B_{s})dB_{s}|^{2}] \rightarrow 0, \ as \
\varepsilon\rightarrow 0.
\end{eqnarray*}
Finally, from equation (\ref{equation1}) it follows that
\begin{eqnarray*}
|B_{t}|=\int_{0}^{t}sgn(B_{s})dB_{s}+L^{0}_{t}.
\end{eqnarray*}
The proof is complete.
\end{proof}\vskip2mm

\noindent {\bf Proof of Lemma \ref{le20}:}
 Now we prove that $sgn(B_{\cdot})\in\widetilde{M}_{\ast}^{2}(0,T)$.
 Let $\pi=\{0=t_{0}<t_{1}\cdots<t_{n}=T\}$, $n\geq1,$ be a
   partition of $[0,T]$ and $B_{t}^{n}=\sum\limits_{j=0}^{n-1}B_{t_{j}}
1_{[t_{j},t_{j+1})}(t)$. Then
\begin{eqnarray*}
&& \mathbb{\hat{E}}[\int_{0}^{t}(\varphi_{n}'(B_{s})-\varphi_{n}'(B^{n}_{s}))^{2}d\langle B \rangle_{s}]\\
&\leq&\dfrac{1}{\varepsilon^{2}}\mathbb{\hat{E}}[\int_{0}^{t}(B_{s}-B^{n}_{s})^{2}d\langle
B \rangle_{s}] \rightarrow 0, \ as \ n\rightarrow \infty.
\end{eqnarray*}
Since $ \varphi'_{n}\rightarrow \varphi'_{\varepsilon}$, uniformly
in $\mathbb{R}$, we have
\begin{eqnarray*}
\mathbb{\hat{E}}[\int_{0}^{t}(\varphi_{\varepsilon}'(B_{s}
)-\varphi_{n}'(B_{s}))^{2}d\langle B \rangle_{s}] \rightarrow 0, \
as \ n\rightarrow \infty.
\end{eqnarray*}
 By virtue of Proposition \ref{pr} we know that
\begin{eqnarray*}
&&\mathbb{\hat{E}}[\int_{0}^{t}(\varphi_{\varepsilon}'(B_{s}
)-sgn(B_{s}))^{2}d\langle B \rangle_{s}]\\
&\leq&C\mathbb{\hat{E}}[\int_{0}^{t}1_{ (-\varepsilon,
\varepsilon)}(B_{s})d\langle B \rangle_{s}]\rightarrow 0, \ as \
\varepsilon\rightarrow 0.
\end{eqnarray*}
Consequently, from the above estimates
\begin{eqnarray*}
 \lim\limits_{n\rightarrow\infty}\mathbb{\hat{E}}
 [\int_{0}^{t}(sgn(B_{s})-\varphi_{n}'(B^{n}_{s}))^{2}d\langle B \rangle_{s}]
 \rightarrow 0, \ as \ \varepsilon\rightarrow 0.
\end{eqnarray*}
\begin{remark}
Similar to the proof of the above theorem, we can obtain
$1_{\{B_{.}>a\}}\in \widetilde{M}_{\ast}^{2}(0,T)$ and
$1_{\{B_{.}\leq a\}} \in \widetilde{M}_{\ast}^{2}(0,T)$.
\end{remark}

\begin{remark}
In analogy to Theorem \ref{th3}, we  obtain two other forms of
Tanaka formula:
\begin{eqnarray*}
(B_{t}-a)^{+}=(-a)^{+}+\int_{0}^{t}1_{\{B_{s}>a\}}dB_{s}+\frac{1}{2}L^{a}_{t},
\end{eqnarray*}
and
\begin{eqnarray*}
(B_{t}-a)^{-}=(-a)^{-}+\int_{0}^{t}1_{\{B_{s}\leq
a\}}dB_{s}+\frac{1}{2}L^{a}_{t}.
\end{eqnarray*}
\end{remark}

 Denis, Hu and Peng  \cite{DHP} obtained an extension of
 Kolmogorov continuity criterion to the framework of nonlinear expectation
 spaces. It will be needed for the study of joint continuity of
 local time for the $G$-Brownian motion $L^{a}_{t}$ in $(t,a)$.

 \begin{lemma} \label{Kolmogorov}
Let $d\geq 1$, $p>0$ and $(X_{t})_{t\in [0,T]^{d}}\subseteq
\mathbb{L}^{p}$ be such that there exist  positive constants $C$ and
$\varepsilon>0$ such that
$$\mathbb{\hat{E}}[|X_{t}-X_{s}|^{p}]\leq C|t-s|^{d+\varepsilon}, \ \text{for all} \ t,s \in[0,T].$$
Then the field $(X_{t})_{t\in [0,T]^{d}}$ admits a continuous
modification $(\text{\~{X}})_{t\in [0,T]^{d}}$  $(i.e.\
\text{\~{X}}_{t}=X_{t}, \ q.s.$,  for all $ t \in[0,T])$  such that
$$\mathbb{\hat{E}}\Big[(\sup\limits_{t\neq s}\frac{|\text{\~{X}}_{t}-\text{\~{X}}_{s}|}
{|t-s|^{\alpha}})^{p}\Big]<+\infty,$$ for every $\alpha\in[0,
\varepsilon/p[$. As a consequence, paths of $\text{\~{X}}$ are
quasi-surely H\"{o}lder continuous of order $\alpha$, for every
$\alpha<\varepsilon/p$.
\end{lemma}

We will show that the local time for the $G$-Brownian motion has a
jointly continuous modification. For the classical case, we refer to
\cite{CW}, \cite{KS} and \cite{RY1999}.  For the proof we use an
approximation method here, which is different from the classical
case.
\begin{theorem}\label{th2}
 For all $t\in[0,T]$,  there exists a jointly
continuous modification of $(a, t)\mapsto L_t^{a}$. Moreover, $(a,
t)\mapsto L_t^{a}$ is H\"{o}lder continuous of order $\gamma$ for
all $\gamma<\frac{1}{2}$.
\end{theorem}

\begin{proof}
From Lemma \ref{le1} we know that, for all  $s,t\geq 0$ and $a,b\in
\mathbb{R}$,
\begin{eqnarray}\label{4}
&&\mathbb{\hat{E}}[(|B_{t}-a|-|B_{s}-b|)^{2p}]\nonumber\\
&\leq&
C_{p}\mathbb{\hat{E}}[|B_{t}-B_{s}|^{2p}]+C_{p}|a-b|^{2p}\nonumber\\
&\leq& C_{p}|t-s|^{p}+C_{p}|a-b|^{2p}
\end{eqnarray}
The constant $C_{p}$ only depends on $p$. By choosing $p>2$, we see
from the generalized Kolmogorov continuity criterion (see Lemma
\ref{Kolmogorov}) that $|B_{t}-a|$ has a jointly continuous
modification in $(a, t)$.

Let us now prove that also the integral
$\int_{0}^{t}sgn(B_{s}-a)dB_{s}$ has a jointly continuous
modification. For this end we let $\delta>0$ and $p>0$. Then

\begin{eqnarray*}\label{}
&&\mathbb{\hat{E}}[|\int_{0}^{t}sgn(B_{s}-a)dB_{s}-\int_{0}^{t}sgn(B_{s}-a-\delta)dB_{s}|^{2p}]\\
&&\leq2^{2p}\mathbb{\hat{E}}[|\int_{0}^{t}1_{[a,
a+\delta]}(B_{s})dB_{s}|^{2p}],
\end{eqnarray*}
and from Lemma \ref{le1} it follows that
\begin{eqnarray}\label{eq1}
&&\mathbb{\hat{E}}[|\int_{0}^{t}sgn(B_{s}-a)dB_{s}-\int_{0}^{t}sgn(B_{s}-a-\delta)dB_{s}|^{2p}]\nonumber\\
&&\leq C_{p}\mathbb{\hat{E}}[|\int_{0}^{t}1_{[a,
a+\delta]}(B_{s})d\langle B \rangle_{ s}|^{p}].
\end{eqnarray}

 We denote $\varphi$ the same as in the proof of Proposition
 \ref{pr}. By applying the generalized It\^o formula for the $G$-Brownian motion
 (see Theorem \ref{th1}) to $\varphi(B_{t})$, we deduce that
\begin{eqnarray*}
\varphi(B_{t})=\int_{0}^{t}\varphi'(B_{s}%
)dB_{s}+\frac{1}{2}\int_{0}^{t}\varphi''(B_{s})d\left \langle
B\right \rangle _{s}.
\end{eqnarray*}
Therefore,
\begin{eqnarray}\label{eq2}
\int_{0}^{t}1_{[a,a+\delta]}(B_{s})d\langle B \rangle_{ s}
&\leq& \delta^{2}\int_{0}^{t}\varphi''(B_{s})d\langle B \rangle_{ s}\nonumber\\
&=&2\delta^{2}\varphi(B_{t})-2\delta^{2}\int_{0}^{t}\varphi'(B_{s}
)dB_{s}\nonumber\\
&\leq& 8\delta|B_{t}|+2\delta^{2}|\int_{0}^{t}\varphi'(B_{s}
)dB_{s}|.
\end{eqnarray}
Hereafter, $C_{p}$ may be different from line to line, but only
depends on $p$. From the inequalities (\ref{eq1}), (\ref{eq2}) and
Lemma \ref{le1} it follows that
\begin{eqnarray}\label{eq3}
&&\mathbb{\hat{E}}[|\int_{0}^{t}sgn(B_{s}-a)dB_{s}-\int_{0}^{t}sgn(B_{s}-a-\delta)dB_{s}|^{2p}]\nonumber\\
&\leq&
C_{p}\mathbb{\hat{E}}[(8\delta|B_{t}|+2\delta^{2}|\int_{0}^{t}\varphi'(B_{s}
)dB_{s}|)^{p}]\nonumber\\
&\leq&
C_{p}\delta^{p}\mathbb{\hat{E}}[|B_{t}|^{p}]+C_{p}\delta^{2p}\mathbb{\hat{E}}[|\int_{0}^{t}\varphi'(B_{s}
)dB_{s}|^{p}]\nonumber\\
&\leq&
C_{p}\delta^{p}\mathbb{\hat{E}}[|B_{t}|^{p}]+C_{p}\delta^{2p}\mathbb{\hat{E}}[(\int_{0}^{t}(\varphi'(B_{s}
))^{2}d\left \langle
B\right \rangle _{s})^{\frac{p}{2}}]\nonumber\\
&\leq&
C_{p}\delta^{p}\mathbb{\hat{E}}[|B_{t}|^{p}]+C_{p}\delta^{p}\mathbb{\hat{E}}[\left
\langle
B\right \rangle _{t}^{\frac{p}{2}}]\nonumber\\
&\leq&C_{p}\delta^{p}\mathbb{\hat{E}}[|B_{T}|^{p}]+C_{p}\delta^{p}\mathbb{\hat{E}}[\left
\langle B\right \rangle _{T}^{\frac{p}{2}}]\nonumber\\
&=&C_{p}\delta^{p}.
\end{eqnarray}
Therefore, for all $0\leq r\leq t$,
\begin{eqnarray}\label{eq3}
&&\mathbb{\hat{E}}[|\int_{0}^{t}sgn(B_{s}-a)dB_{s}-\int_{0}^{r}sgn(B_{s}-a-\delta)dB_{s}|^{2p}]\nonumber\\
&\leq& C_{p}\mathbb{\hat{E}}[|\int_{0}^{r}sgn(B_{s}-a)dB_{s}-\int_{0}^{r}sgn(B_{s}-a-\delta)dB_{s}|^{2p}]\nonumber\\
&&+C_{p}\mathbb{\hat{E}}[|\int_{r}^{t}sgn(B_{s}-a)dB_{s}|^{2p}]\nonumber\\
&\leq&
C_{p}\delta^{p}+C_{p}\mathbb{\hat{E}}[(\int_{r}^{t}(sgn(B_{s}-a))^{2}d\langle B\rangle_{s})^{p}]\nonumber\\
&\leq& C_{p}\delta^{p}+C_{p}|t-r|^{p}.
\end{eqnarray}
 We choose $p>2$, then by the generalized Kolmogorov
continuity criterion (Lemma \ref{Kolmogorov}), we obtain the
existence of a jointly continuous modification of $(\alpha,
t)\mapsto L_t^{\alpha}$. The proof is complete.
\end{proof}

\begin{corollary}\label{co2}
 For all $t\in [0,T], a\in\mathbb{R}$, we have
 $$\int_{0}^{t}sgn(B_{s}-a)dB_{s},
 \int_{0}^{t}1_{[a,\infty)}(B_{s})dB_{s},
 \int_{0}^{t}1_{(-\infty,a)}(B_{s})dB_{s}$$ have a jointly
continuous modification.
\end{corollary}

\section{Quadratic variation of local time for the $G$-Brownian motion}
The objective of this section is to  study the quadratic variation
of the local time for the $G$-Brownian motion. For this end, we
begin with the following Lemma:
\begin{lemma}\label{le2}
 If $f:\mathbb{R}\rightarrow[0,\infty)$ is a continuous function with compact
 support, then, for all $t\geq 0$, we have
 \begin{eqnarray*}\label{}
\int_{-\infty}^{\infty}f(a)\Big(\int_{0}^{t}sgn(B_{s}-a)dB_{s}\Big)da
=\int_{0}^{t}\Big(\int_{-\infty}^{\infty}f(a)sgn(B_{s}-a)da\Big)dB_{s},
q.s.,
\end{eqnarray*}
and
\begin{eqnarray*}\label{}
\int_{-\infty}^{\infty}f(a)\Big(\int_{0}^{t}1_{[a,\infty)}(B_{s})dB_{s}\Big)da
=\int_{0}^{t}\Big(\int_{-\infty}^{\infty}f(a)1_{[a,\infty)}(B_{s})da\Big)dB_{s},
q.s.
\end{eqnarray*}
\end{lemma}

\begin{proof}
   We only prove the fist equality, the second equality can be proved in a similar way.
   Without loss of generality we can assume that $f$ has its support in
$[0,1]$. Let
\begin{eqnarray*}\label{}
\varphi_{n}(x)=\sum\limits_{k=0}^{2^{n}-1}\frac{1}{2^{n}}f(\frac{k}{2^{n}})sgn(x-\frac{k}{2^{n}}),
x\in\mathbb{R}, n\geq 1.
\end{eqnarray*}
Then
\begin{eqnarray*}\label{}
\int_{0}^{t}\varphi_{n}(B_{s})dB_{s}
=\sum\limits_{k=0}^{2^{n}-1}\frac{1}{2^{n}}f(\frac{k}{2^{n}})\int_{0}^{t}sgn(B_{s}-\frac{k}{2^{n}})dB_{s}.
\end{eqnarray*}
 By the proof Theorem \ref{th2}  we know that
  the integral $\int_{0}^{t}sgn(B_{s}-a)dB_{s}$ is jointly continuous in $(a,
t)$. Therefore,
\begin{eqnarray}\label{eq4}
\int_{0}^{1}\Big(\int_{0}^{t}f(a)sgn(B_{s}-a)dB_{s}\Big)da
=\lim\limits_{n\rightarrow\infty}\int_{0}^{t}\varphi_{n}(B_{s})dB_{s},
\ q.s.
\end{eqnarray}
From Lemma \ref{le1} and the convergence of $\varphi_{n}(x) $
 to $\int_{0}^{1}f(a)sgn(x-a)da$, uniformly w.r.t $x\in\mathbb{R}$, we
 obtain
\begin{eqnarray*}\label{}
&&\mathbb{\hat{E}}[|\int_{0}^{t}\varphi_{n}(B_{s})dB_{s}
-\int_{0}^{t}\Big(\int_{0}^{1}f(a)sgn(B_{s}-a)da\Big)dB_{s}|^{2}]\\
&\leq&C\mathbb{\hat{E}}[\int_{0}^{t}\Big(\varphi_{n}(B_{s})-\int_{0}^{1}f(a)sgn(B_{s}-a)da\Big)^{2}ds]\\
 &\rightarrow & 0, \ \text{as}\ n\rightarrow\infty.
\end{eqnarray*}
Using Proposition \ref{pr1} we can deduce the existence of  a
subsequence $\{\varphi_{n_{k}}\}_{k=1}^{\infty}$ such that
\begin{eqnarray}\label{eq5}
\int_{0}^{t}\Big(\int_{0}^{1}f(a)sgn(B_{s}-a)da\Big)dB_{s}
=\lim\limits_{k\rightarrow\infty}\int_{0}^{t}\varphi_{n_{k}}(B_{s})dB_{s},
\ q.s.
\end{eqnarray}
Finally, the relations (\ref{eq4}) and (\ref{eq5}) yield
 \begin{eqnarray*}\label{}
\int_{0}^{1}f(a)\Big(\int_{0}^{t}sgn(B_{s}-a)dB_{s}\Big)da
=\int_{0}^{t}\Big(\int_{0}^{1}f(a)sgn(B_{s}-a)da\Big)dB_{s},\ q.s.
\end{eqnarray*}
 The proof is complete.
\end{proof}\vskip3mm

We now establish an occupation time formula.  For the classical
case, we refer to \cite{CW}, \cite{KS} and \cite{RY1999}.
\begin{theorem}\label{th10}
 For all $t\geq 0$ and all reals $a\leq b$, we have
 \begin{eqnarray*}\label{}
\int_{0}^{t}1_{(a,b)}(B_{s})d\langle
B\rangle_{s}=\int_{a}^{b}L_{t}^{x}dx, \ q.s.
\end{eqnarray*}
\end{theorem}

\begin{proof}
 Put $\varphi_{x,\varepsilon}(y):=\varphi_{\varepsilon}(y-x)$, where
 $\varphi_{\varepsilon}$ is the function in the proof of Theorem
 \ref{th3}. In analogy to equality (\ref{equation1}) we have
\begin{eqnarray}\label{equation10}
\varphi_{x,\varepsilon}(B_{t})=\varphi_{x,\varepsilon}(B_{0})
+\int_{0}^{t}\varphi_{x,\varepsilon}'(B_{s}
)dB_{s}+\frac{1}{2\varepsilon}\int_{0}^{t}1_{ (x-\varepsilon,
x+\varepsilon)}(B_{s})d\langle B \rangle_{s},
\end{eqnarray}
where
$$\varphi_{x,\varepsilon}'(z)=\dfrac{1}{\varepsilon}\int_{x-\varepsilon}^{x+\varepsilon}
1_{[y,\infty)}(z)dy-1.$$
 Thanks to Lemma \ref{le2}, we have
\begin{eqnarray*}\label{}
\int_{0}^{t}\varphi_{x,\varepsilon}'(B_{s})dB_{s}
&=&\dfrac{1}{\varepsilon}\int_{0}^{t}\int_{x-\varepsilon}^{x+\varepsilon}
1_{[y,\infty)}(B_{s})dydB_{s}-B_{t}\\
&=&\dfrac{1}{\varepsilon}\int_{x-\varepsilon}^{x+\varepsilon}\int_{0}^{t}
1_{[y,\infty)}(B_{s})dB_{s}dy-B_{t}.
\end{eqnarray*}
Therefore, by equality (\ref{equation10}) we have
\begin{eqnarray}\label{eq6}
&&\int_{a}^{b}\Big(\varphi_{x,\varepsilon}(B_{t})-\varphi_{x,\varepsilon}(B_{0})
-\dfrac{1}{\varepsilon}\int_{x-\varepsilon}^{x+\varepsilon}\int_{0}^{t}
1_{[y,\infty)}(B_{s})dB_{s}dy+B_{t}\Big)dx\nonumber\\
&&=\frac{1}{2\varepsilon}\int_{a}^{b}\int_{0}^{t}1_{ (x-\varepsilon,
x+\varepsilon)}(B_{s})d\langle B \rangle_{s}dx.
\end{eqnarray}
According to Corollary \ref{co2}, we obtain that
\begin{eqnarray}\label{eq7}
&&\lim\limits_{\varepsilon\downarrow
0}\int_{a}^{b}\Big(\varphi_{x,\varepsilon}(B_{t})-\varphi_{x,\varepsilon}(B_{0})
-\dfrac{1}{\varepsilon}\int_{x-\varepsilon}^{x+\varepsilon}\int_{0}^{t}
1_{[y,\infty)}(B_{s})dB_{s}dy+B_{t}\Big)dx\nonumber\\
&&=\int_{a}^{b}\Big(|B_{t}-x|-|x|-2\int_{0}^{t}
1_{[x,\infty)}(B_{s})dB_{s}+B_{t}\Big)dx\nonumber\\
&&=\int_{a}^{b}\Big(|B_{t}-x|-|x|-\int_{0}^{t}sgn(B_{s}-x)dB_{s}\Big)dx,\
q.s.
\end{eqnarray}

For $z\in\mathbb{R}$, we have
\begin{eqnarray*}\label{}
\lim\limits_{\varepsilon\downarrow
0}\dfrac{1}{2\varepsilon}\int_{a}^{b}1_{ (x-\varepsilon,
x+\varepsilon)}(z)dx=1_{ (a, b)}(z)+\dfrac{1}{2}1_{ \{
a\}}(z)+\dfrac{1}{2}1_{ \{ b\}}(z).
\end{eqnarray*}
Therefore,
\begin{eqnarray}\label{eq8}
&&\lim\limits_{\varepsilon\downarrow
0}\frac{1}{2\varepsilon}\int_{a}^{b}\int_{0}^{t}1_{ (x-\varepsilon,
x+\varepsilon)}(B_{s})d\langle B \rangle_{s}dx\nonumber\\
&&=\lim\limits_{\varepsilon\downarrow
0}\frac{1}{2\varepsilon}\int_{0}^{t}\int_{a}^{b}1_{ (x-\varepsilon,
x+\varepsilon)}(B_{s})dxd\langle B \rangle_{s}\nonumber\\
&&=\int_{0}^{t}1_{(a,b)}(B_{s})d\langle
B\rangle_{s}+\frac{1}{2}\int_{0}^{t}1_{\{a,b\}}(B_{s})d\langle
B\rangle_{s}, \ q.s.
\end{eqnarray}
From Corollary \ref{co1} we know that
\begin{eqnarray}\label{eq9}
\int_{0}^{t}1_{\{a, b \}}(B_{s})d\langle B \rangle_{ s}=0,\ q.s.
\end{eqnarray}
Consequently, from (\ref{eq6}), (\ref{eq7}), (\ref{eq8}) and
(\ref{eq9}) it follows that
\begin{eqnarray*}
\int_{a}^{b}\Big(|B_{t}-x|-|x|-\int_{0}^{t}sgn(B_{s}-x)dB_{s}\Big)dx=\int_{0}^{t}1_{(a,b)}(B_{s})d\langle
B\rangle_{s},\ q.s.
\end{eqnarray*}
Finally, thanks to Theorem \ref{th3}, we have
\begin{eqnarray*}
\int_{a}^{b}L_{t}^{x}dx=\int_{0}^{t}1_{(a,b)}(B_{s})d\langle
B\rangle_{s},\ q.s.
\end{eqnarray*}
 The proof is complete.
\end{proof}\vskip3mm

We now study the quadratic variation of the local time for the
$G$-Brownian motion. For simplicity, we denote, for $t\geq 0$ and
$x\in\mathbb{R}$,
 \begin{eqnarray*}\label{}
&&Y_{t}^{x} =\int_{0}^{t}sgn(B_{s}-x)dB_{s},\\
&& N_{t}^{x} =|B_{t}-x|-|x|.
\end{eqnarray*}
Then, from Theorem \ref{th3} we have
\begin{eqnarray*}\label{}
L_{t}^{x}= N_{t}^{x} -Y_{t}^{x}.
\end{eqnarray*}

We first give the following lemma, which is an immediate consequence
of the proof of Theorem \ref{th2}.
\begin{lemma}\label{le30}
 For every
$p\geq1$, there exists a constant $C=C(p)>0$ such that
 \begin{eqnarray*}\label{}
\mathbb{\hat{E}}[|N_{t}^{x} -N_{s}^{y}|^{p}]\leq C(|x -y|^{p}+|t
-s|^{\frac{p}{2}}),
\end{eqnarray*}
 \begin{eqnarray*}\label{}
\mathbb{\hat{E}}[|Y_{t}^{x}-Y_{t}^{y}|^{p}]\leq C|x
-y|^{\frac{p}{2}},
\end{eqnarray*}
 for all $t,s\in[0,T]$ and $x,y\in\mathbb{R}$.
\end{lemma}

 Now we give the main result in this section.

\begin{theorem}\label{}
 Let $\underline{\sigma}>0$. Then for all $p\geq1$ and $a\leq b$, we have
 along the sequence of partitions  of  $\pi_{n}=\{a_{i}^{n}=a+\dfrac{i(b-a)}{2^{n}}, \
i=0,1,\cdots,2^{n}\}, n\geq 1$, of the interval of $[a,b]$  the
following convergence
 \begin{eqnarray*}\label{}
\lim\limits_{n\rightarrow\infty}\sum\limits_{i=0}^{2^{n}-1}(L_{t}^{a_{i+1}^{n}}-L_{t}^{a_{i}^{n}})^{2}
=4\int_{a}^{b}L_{t}^{x}dx, \ \text{in} \ \mathbb{L}^{p},
\end{eqnarray*}
uniformly with respect to $t\in[0,T]$.
\end{theorem}

\begin{proof}
By applying the generalized It\^o formula for the $G$-Brownian
motion (cf. Theorem \ref{th1}) we obtain
 \begin{eqnarray*}\label{}
&&\sum\limits_{i=0}^{2^{n}-1}(L_{t}^{a_{i+1}^{n}}-L_{t}^{a_{i}^{n}})^{2}
-4\int_{a}^{b}L_{t}^{x}dx\\
&=&\sum\limits_{i=0}^{2^{n}-1}(N_{t}^{a_{i+1}^{n}}-N_{t}^{a_{i}^{n}}-Y_{t}^{a_{i+1}^{n}}+Y_{t}^{a_{i}^{n}})^{2}
-4\int_{a}^{b}L_{t}^{x}dx\\
&=&\sum\limits_{i=0}^{2^{n}-1}\Big((N_{t}^{a_{i+1}^{n}}-N_{t}^{a_{i}^{n}})^{2}
-2(N_{t}^{a_{i+1}^{n}}-N_{t}^{a_{i}^{n}})(Y_{t}^{a_{i+1}^{n}}-Y_{t}^{a_{i}^{n}})\\&&
+(Y_{t}^{a_{i+1}^{n}}-Y_{t}^{a_{i}^{n}})^{2}\Big)-4\int_{a}^{b}L_{t}^{x}dx\\
&=&\sum\limits_{i=0}^{2^{n}-1}\Big((N_{t}^{a_{i+1}^{n}}-N_{t}^{a_{i}^{n}})^{2}
-2(N_{t}^{a_{i+1}}-N_{t}^{a_{i}^{n}})(Y_{t}^{a_{i+1}^{n}}-Y_{t}^{a_{i}^{n}})\\&&
+4\int_{0}^{t}(Y_{s}^{a_{i+1}^{n}}-Y_{s}^{a_{i}^{n}})1_{(a_{i}^{n},a_{i+1}^{n})}(B_{s})dB_{s}+
4\int_{0}^{t} 1_{(a_{i}^{n},a_{i+1}^{n})}(B_{s})d\langle
B\rangle_{s}\Big)-4\int_{a}^{b}L_{t}^{x}dx.
\end{eqnarray*}
Thus, thanks to Theorem \ref{th10}, we get
 \begin{eqnarray}\label{eq10}
&&\sum\limits_{i=0}^{2^{n}-1}(L_{t}^{a_{i+1}^{n}}-L_{t}^{a_{i}^{n}})^{2}
-4\int_{a}^{b}L_{t}^{x}dx\nonumber\\
&=&\sum\limits_{i=0}^{2^{n}-1}(N_{t}^{a_{i+1}^{n}}-N_{t}^{a_{i}^{n}})^{2}
-2\sum\limits_{i=0}^{2^{n}-1}(N_{t}^{a_{i+1}^{n}}-N_{t}^{a_{i}^{n}})(Y_{t}^{a_{i+1}^{n}}-Y_{t}^{a_{i}^{n}})
\nonumber\\&&+4\sum\limits_{i=0}^{2^{n}-1}\int_{0}^{t}
(Y_{s}^{a_{i+1}^{n}}-Y_{s}^{a_{i}^{n}})1_{(a_{i}^{n},a_{i+1}^{n})}(B_{s})dB_{s}.
\end{eqnarray}
Then, from Lemma \ref{le30} and the subadditivity of the sublinear
expectation $\mathbb{\hat{E}}[\cdot]$ it follows that
 \begin{eqnarray}\label{eq11}
\mathbb{\hat{E}}[\Big(\sum\limits_{i=0}^{2^{n}-1}(N_{t}^{a_{i+1}^{n}}-N_{t}^{a_{i}^{n}})^{2}\Big)^{p}]
&\leq&
2^{n(p-1)}\sum\limits_{i=0}^{2^{n}-1}\mathbb{\hat{E}}[(N_{t}^{a_{i+1}^{n}}-N_{t}^{a_{i}^{n}})^{2p}]\nonumber\\
&\leq&C  2^{n(p-1)} 2^{n} 2^{-2np}= C 2^{-np}.
\end{eqnarray}
Moreover, by virtue of  Lemma \ref{le30}, the subadditivity of the
sublinear expectation $\mathbb{\hat{E}}[\cdot]$ and H\"{o}lder
inequality for the sublinear expectation (see \cite{Peng2008}) we
can estimate the second term at the right hand of (\ref{eq10}) as
follows:
 \begin{eqnarray}\label{eq12}
&&\mathbb{\hat{E}}[\Big(\sum\limits_{i=0}^{2^{n}-1}(N_{t}^{a_{i+1}^{n}}
-N_{t}^{a_{i}^{n}})(Y_{t}^{a_{i+1}^{n}}-Y_{t}^{a_{i}^{n}})\Big)^{p}]\nonumber\\
&\leq& 2^{n(p-1)}\sum\limits_{i=0}^{2^{n}-1}
\mathbb{\hat{E}}[|(N_{t}^{a_{i+1}^{n}}-N_{t}^{a_{i}^{n}})(Y_{t}^{a_{i+1}^{n}}-Y_{t}^{a_{i}^{n}})|^{p}]\nonumber\\
&\leq&2^{n(p-1)}\sum\limits_{i=0}^{2^{n}-1}
\Big(\mathbb{\hat{E}}[|N_{t}^{a_{i+1}^{n}}-N_{t}^{a_{i}^{n}}|^{2p}]\Big)^{\frac{1}{2}}
\Big(\mathbb{\hat{E}}[|Y_{t}^{a_{i+1}^{n}}-Y_{t}^{a_{i}^{n}}|^{2p}]\Big)^{\frac{1}{2}}\nonumber\\
&\leq&C  2^{n(p-1)} 2^{n} 2^{-np}2^{-\frac{n}{2}p}= C
2^{-\frac{np}{2}}.
\end{eqnarray}
Here, for the latter estimate we have used Lemma \ref{le30}.

 Finally, let us estimate the third term at the right hand of (\ref{eq10}).
 Using again the subadditivity of the sublinear expectation
 $\mathbb{\hat{E}}[\cdot]$ and  H\"{o}lder inequality for the sublinear expectation as well as Lemma \ref{le1} we
 have
\begin{eqnarray*}\label{}
&&\mathbb{\hat{E}}[\Big(\sum\limits_{i=0}^{2^{n}-1}
\int_{0}^{t}(Y_{t}^{a_{i+1}^{n}}-Y_{t}^{a_{i}^{n}})1_{(a_{i}^{n},a_{i+1}^{n})}(B_{s})dB_{s}\Big)^{p}]
\nonumber\\
&\leq& C \mathbb{\hat{E}}[\Big(\int_{0}^{t}\big(\sum\limits_{i=0}^{2^{n}-1}
(Y_{t}^{a_{i+1}^{n}}-Y_{t}^{a_{i}^{n}})1_{(a_{i}^{n},a_{i+1}^{n})}(B_{s})\big)^{2}ds\Big)^{\frac{p}{2}}]\nonumber\\
&\leq&C\mathbb{\hat{E}}[\int_{0}^{t}\Big(\sum\limits_{i=0}^{2^{n}-1}
|Y_{t}^{a_{i+1}^{n}}-Y_{t}^{a_{i}^{n}}|1_{(a_{i}^{n},a_{i+1}^{n})}(B_{s})\Big)^{p}ds]\nonumber\\
&=&C\mathbb{\hat{E}}[\int_{0}^{t}\sum\limits_{i=0}^{2^{n}-1}
|Y_{t}^{a_{i+1}^{n}}-Y_{t}^{a_{i}^{n}}|^{p}1_{(a_{i}^{n},a_{i+1}^{n})}(B_{s})ds]\nonumber\\
&\leq&C\sum\limits_{i=0}^{2^{n}-1}\Big(\mathbb{\hat{E}}[\int_{0}^{t}
|Y_{t}^{a_{i+1}^{n}}-Y_{t}^{a_{i}^{n}}|^{2p}ds]\Big)^{\frac{1}{2}}\Big(\mathbb{\hat{E}}[\int_{0}^{t}
1_{(a_{i}^{n},a_{i+1}^{n})}(B_{s})ds]\Big)^{\frac{1}{2}}.\nonumber
\end{eqnarray*}
Consequently, thanks to Lemma \ref{le30} and Proposition \ref{pr},
we have
\begin{eqnarray}\label{eq13}
&&\mathbb{\hat{E}}[\Big(\sum\limits_{i=0}^{2^{n}-1}
\int_{0}^{t}(Y_{t}^{a_{i+1}^{n}}-Y_{t}^{a_{i}^{n}})1_{(a_{i}^{n},a_{i+1}^{n})}(B_{s})dB_{s}\Big)^{p}]
\nonumber\\
&&\leq C  2^{n} 2^{-\frac{np}{2}}2^{-\frac{n}{2}}= C
2^{-\frac{n(p-1)}{2}}.
\end{eqnarray}
Finally, by substituting the estimates (\ref{eq11}),(\ref{eq12}) and
(\ref{eq13}) for the right hand of (\ref{eq10}), we obtain
\begin{eqnarray*}\label{}
&&\mathbb{\hat{E}}[\Big(\sum\limits_{i=0}^{2^{n}-1}(L_{t}^{a_{i+1}}-L_{t}^{a_{i}})^{2}
-4\int_{a}^{b}L_{t}^{x}dx\Big)^{p}]\\
&&\leq C(2^{-np}+2^{-\frac{np}{2}}+
2^{-\frac{n(p-1)}{2}})\rightarrow 0, \ \text{as} \
n\rightarrow\infty,
\end{eqnarray*}
uniformly with respect to $t\in[0,T]$. The proof is complete.
\end{proof}

\section{A generalized It\^{o} formula for convex functions}
In this section,  our objective is to  obtain a generalization of
It\^{o}'s formula for the $G$-Brownian motion with help of the
Tanaka formula for the $G$-Brownian motion, which generalizes the
corresponding result in \cite{Peng2009}.

 Given a convex function $f:\mathbb{R}\rightarrow\mathbb{R}$, we associate with the
 measure $\mu$ on $\mathcal {B}(\mathbb{R})$ which is defined as follows:
 \begin{eqnarray*}\label{}
\mu[a,b)=f'_{-}(b)-f'_{-}(a),\ \text{for all}\ a< b.
\end{eqnarray*}
Then, for  $\varphi\in C_{K}^{2}(\mathbb{R})$, we have
 \begin{eqnarray*}\label{}
\int_{\mathbb{R}}\varphi''(x)f(x)dx=\int_{\mathbb{R}}\varphi(x)\mu(dx).
\end{eqnarray*}
This relation allows to identify the generalized second derivative
of $f$ with the measure $\mu$.

Now we establish the main result in this section.
\begin{theorem}\label{th4}
 Let $f:\mathbb{R}\rightarrow\mathbb{R}$ be a convex function and assume that its right derivative
 $f'_{+}$ is bounded and $f'_{-}(B_{\cdot}) \in M_{\omega}^{2}(0,T)$. Then, for all $t\geq 0$,
 \begin{eqnarray*}\label{}
f(B_{t})= f(0)+
\int_{0}^{t}f'_{-}(B_{s})dB_{s}+\int_{-\infty}^{\infty}L_{t}^{a}f''(da),
\end{eqnarray*}
where $f''=\mu$ is the measure introduced above.
\end{theorem}

Before giving the proof of Theorem \ref{th4}, we establish the
following lemmas, which will be needed later.

\begin{lemma}\label{le4}
 If $f:\mathbb{R}\rightarrow[0,\infty)$ is a convex function and $\varphi\in C_{b}^{1}(\mathbb{R})$,
 then for all $t\geq 0$, we have
 \begin{eqnarray*}\label{}
\int_{-\infty}^{\infty}\Big(\int_{0}^{t}\varphi(B_{s}-a)dB_{s}\Big)f''(da)
=\int_{0}^{t}\Big(\int_{-\infty}^{\infty}\varphi(B_{s}-a)f''(da)\Big)dB_{s}.
\end{eqnarray*}
\end{lemma}

\begin{proof}
   We put
\begin{eqnarray*}\label{}
\Phi(x)=\int_{0}^{x}\varphi(y)dy, x\in\mathbb{R}.
\end{eqnarray*}
Then, by applying the generalized It\^o formula for the $G$-Brownian
motion (cf. Theorem \ref{th1}) to $\Phi(B_{t}-a)$, we deduce that
\begin{eqnarray*}
\Phi(B_{t}-a)=\Phi(-a)+\int_{0}^{t}\varphi(B_{s}-a
)dB_{s}+\frac{1}{2}\int_{0}^{t}\varphi'(B_{s}-a)d\left \langle
B\right \rangle _{s}.
\end{eqnarray*}
Therefore,
 \begin{eqnarray}\label{eq21}
&&\int_{-\infty}^{\infty}\Big(\int_{0}^{t}\varphi(B_{s}-a)dB_{s}\Big)f''(da)\nonumber\\
&=&\int_{-\infty}^{\infty}\Phi(B_{t}-a)f''(da)-\Phi(-a)\int_{-\infty}^{\infty}f''(da)-\frac{1}{2}\int_{-\infty}^{\infty}\Big(\int_{0}^{t}\varphi'(B_{s}-a)d\left
\langle B\right \rangle _{s}\Big)f''(da)\nonumber\\
&=&\int_{-\infty}^{\infty}\Phi(B_{t}-a)f''(da)-\Phi(-a)\int_{-\infty}^{\infty}f''(da)\nonumber\\
&&-\frac{1}{2}\int_{0}^{t}\Big(\int_{-\infty}^{\infty}\varphi'(B_{s}-a)f''(da)\Big)d\left
\langle B\right \rangle _{s}.
\end{eqnarray}
Let
\begin{eqnarray*}\label{}
g(x)=\int_{-\infty}^{\infty}\Phi(x-a)f''(da).
\end{eqnarray*}
Then, using again the generalized It\^o formula for the $G$-Brownian
motion (see Theorem \ref{th1}) to $g(B_{t})$, we deduce that
\begin{eqnarray*}
g(B_{t})=g(0)+\int_{0}^{t}g'(B_{s}
)dB_{s}+\frac{1}{2}\int_{0}^{t}g''(B_{s})d\left \langle B\right
\rangle _{s}.
\end{eqnarray*}
Consequently,
 \begin{eqnarray}\label{eq22}
\int_{0}^{t}\Big(\int_{-\infty}^{\infty}\varphi(B_{s}-a)f''(da)\Big)dB_{s}
&=&\int_{-\infty}^{\infty}\Phi(B_{t}-a)f''(da)-\Phi(-a)\int_{-\infty}^{\infty}f''(da)\nonumber\\
&-&\frac{1}{2}\int_{0}^{t}\Big(\int_{-\infty}^{\infty}\varphi'(B_{s}-a)f''(da)\Big)d\left
\langle B\right \rangle_{s}.
\end{eqnarray}
Finally, (\ref{eq21}) and (\ref{eq22}) allow to conclude that
 \begin{eqnarray*}\label{}
\int_{-\infty}^{\infty}\Big(\int_{0}^{t}\varphi(B_{s}-a)dB_{s}\Big)f''(da)
=\int_{0}^{t}\Big(\int_{-\infty}^{\infty}\varphi(B_{s}-a)f''(da)\Big)dB_{s}.
\end{eqnarray*}
 The proof is complete.
\end{proof}

\begin{lemma}\label{le5}
 If $f:\mathbb{R}\rightarrow[0,\infty)$ is a convex function and
 its right derivative  $f'_{+}$ is bounded, then for all $t\geq 0$, we have
  \begin{eqnarray*}\label{}
\int_{-\infty}^{\infty}\Big(\int_{0}^{t}sgn(B_{s}-a)dB_{s}\Big)f''(da)
=\int_{0}^{t}\Big(\int_{-\infty}^{\infty}sgn(B_{s}-a)f''(da)\Big)dB_{s},
q.s.
\end{eqnarray*}
\end{lemma}

\begin{proof}
   For  $\varepsilon>0$, we define $C^{1}-$ function
 $\varphi_{\varepsilon}:\mathbb{R}\rightarrow\mathbb{R}$ such that $ |\varphi_{\varepsilon}|\leq
 1$,  and
\[
\varphi_{\varepsilon}(x)=\left\{
\begin{array}{lll}
-1; & if\ x \leq -\varepsilon,\\
1; & if\ x \geq \varepsilon.\\
\end{array}
\right.
\]
Then, we have $|\varphi_{\varepsilon}(x)-sgn(x)|\leq2
1_{(-\varepsilon,\varepsilon)}(x)$, for $x\in\mathbb{R}$. From Lemma
\ref{le4}  we have
 \begin{eqnarray*}\label{}
\int_{-\infty}^{\infty}\Big(\int_{0}^{t}\varphi_{\varepsilon}(B_{s}-a)dB_{s}\Big)f''(da)
=\int_{0}^{t}\Big(\int_{-\infty}^{\infty}\varphi_{\varepsilon}(B_{s}-a)f''(da)\Big)dB_{s}.
\end{eqnarray*}
From Lemma \ref{le1} and Proposition \ref{pr} it then follows that
\begin{eqnarray*}\label{}
&&\mathbb{\hat{E}}[|\int_{0}^{t}\Big(\int_{-\infty}^{\infty}\varphi_{\varepsilon}(B_{s}-a)f''(da)\Big)dB_{s}-
\int_{0}^{t}\Big(\int_{-\infty}^{\infty}sgn(B_{s}-a)f''(da)\Big)dB_{s}|]\\
&\leq&\mathbb{\hat{E}}[\int_{0}^{t}\Big(\int_{-\infty}^{\infty}\varphi_{\varepsilon}(B_{s}-a)f''(da)
-\int_{-\infty}^{\infty}sgn(B_{s}-a)f''(da)\Big)^{2}d\left\langle B\right \rangle _{s}]^{\frac{1}{2}}\\
&\leq&\mathbb{\hat{E}}[\int_{0}^{t}\Big(\int_{-\infty}^{\infty}|\varphi_{\varepsilon}(B_{s}-a)
-sgn(B_{s}-a)|f''(da)\Big)^{2}d\left\langle B\right \rangle _{s}]^{\frac{1}{2}}\\
&\leq&C\mathbb{\hat{E}}[\int_{0}^{t}\Big(\int_{-\infty}^{\infty}1_{\{-\varepsilon,\varepsilon\}}(B_{s}-a)f''(da)\Big)^{2}d\left\langle B\right \rangle _{s}]^{\frac{1}{2}}\\
&\leq&C \Big(\int_{-\infty}^{\infty}f''(da)\Big)^{\frac{1}{2}}
\Big(\int_{-\infty}^{\infty}\mathbb{\hat{E}}[\int_{0}^{t}1_{\{-\varepsilon,\varepsilon\}}(B_{s}-a)
d\left\langle B\right \rangle _{s}]f''(da)\Big)^{\frac{1}{2}}\\
&\leq&C\varepsilon^{\frac{1}{2}}\int_{-\infty}^{\infty}f''(da)\rightarrow
 0, \ \text{as}\ \varepsilon \rightarrow 0,
\end{eqnarray*}
and
\begin{eqnarray*}\label{}
&&\mathbb{\hat{E}}[|\int_{-\infty}^{\infty}\Big(\int_{0}^{t}\varphi_{\varepsilon}(B_{s}-a)dB_{s}\Big)f''(da)-
\int_{-\infty}^{\infty}\Big(\int_{0}^{t}sgn(B_{s}-a)dB_{s}\Big)f''(da)|]\\
&\leq&\int_{-\infty}^{\infty}\mathbb{\hat{E}}[|\int_{0}^{t}\Big(\varphi_{\varepsilon}(B_{s}-a)-
sgn(B_{s}-a)\Big)dB_{s}|]f''(da)\\
&\leq&\int_{-\infty}^{\infty}\mathbb{\hat{E}}[|\int_{0}^{t}\Big(\varphi_{\varepsilon}(B_{s}-a)-
sgn(B_{s}-a)\Big)^{2}d\left\langle B\right \rangle _{s}|]^{\frac{1}{2}}f''(da)\\
&\leq&C\int_{-\infty}^{\infty}\mathbb{\hat{E}}[\int_{0}^{t}1_{\{-\varepsilon,\varepsilon\}}(B_{s}-a)
d\left\langle B\right \rangle _{s}]^{\frac{1}{2}}f''(da)\\
&\leq&C\varepsilon^{\frac{1}{2}}\int_{-\infty}^{\infty}f''(da)\rightarrow
 0, \ \text{as}\ \varepsilon \rightarrow 0,
\end{eqnarray*}
Consequently,
\begin{eqnarray*}\label{}
\int_{-\infty}^{\infty}\Big(\int_{0}^{t}sgn(B_{s}-a)dB_{s}\Big)f''(da)
=\int_{0}^{t}\Big(\int_{-\infty}^{\infty}sgn(B_{s}-a)f''(da)\Big)dB_{s},
q.s.
\end{eqnarray*}
 The proof is complete.
\end{proof}\vskip2mm

Now we give the proof of Theorem \ref{th4}. \vskip2mm

\begin{proof}
  Since the function $f'_{+}$ is  bounded, the measure $f''$ has its support in a
  compact interval $I$. From Section 3 in Appendix in \cite{RY1999}, we
  know that
\begin{eqnarray}\label{eq6}
f(x)=\dfrac{1}{2}\int_{I}|x-a|f''(a)da+\alpha_{I}x+\beta_{I},
\end{eqnarray}
where $\alpha_{I}$ and $\beta_{I}$ are constants. Moreover, for all
$x$ outside of a countable set $\Gamma$ in $\mathbb{R}$,
\begin{eqnarray}\label{eq7}
f'_{-}(x)=\dfrac{1}{2}\int_{I}sgn(x-a)|f''(da)+\alpha_{I}.
\end{eqnarray}
However, due to Corollary \ref{co1},
\begin{eqnarray*}
\int_{0}^{t}1_{\Gamma}(B_{s})d\langle B \rangle_{ s}=0,\ q.s.,t\geq
0.
\end{eqnarray*}
Consequently, thanks to the  equalities (\ref{eq6}) and (\ref{eq7}),
for all $t\geq 0$,  we have
\begin{eqnarray}\label{eq8}
f(B_{t})=\dfrac{1}{2}\int_{I}|B_{t}-a|f''(da)+\alpha_{I}B_{t}+\beta_{I},\
q.s.,
\end{eqnarray}
and
\begin{eqnarray}\label{eq9}
f'_{-}(B_{t})=\dfrac{1}{2}\int_{I}sgn(B_{t}-a)|f''(da)+\alpha_{I},\
q.s.
\end{eqnarray}
On the other hand, from Theorem \ref{th3} and equality (\ref{eq8})
we deduce that
\begin{eqnarray*}\label{}
f(B_{t})&=&\dfrac{1}{2}\int_{I}\Big(|a|+\int_{0}^{t}sgn(B_{s}-a)dB_{s}+L^{a}_{t}\Big)f''(da)+\alpha_{I}B_{t}+\beta_{I}\\
&=&f(0)+\dfrac{1}{2}\int_{I}\Big(\int_{0}^{t}sgn(B_{s}-a)dB_{s}\Big)f''(da)
+\dfrac{1}{2}\int_{I}L^{a}_{t}f''(da)+\alpha_{I}B_{t}.
\end{eqnarray*}
Finally, thanks to Lemma \ref{le5} and equality (\ref{eq9}), we have
\begin{eqnarray*}\label{}
f(B_{t})&=&f(0)+\dfrac{1}{2}\int_{0}^{t}\Big(\int_{I}sgn(B_{s}-a)f''(da)\Big)dB_{s}
+\dfrac{1}{2}\int_{I}L^{a}_{t}f''(da)+\alpha_{I}B_{t}\\
&=&f(0)+\int_{0}^{t}f'_{-}(B_{s})dB_{s}+\int_{I}L_{t}^{a}f''(da).
\end{eqnarray*}
 The proof is complete.
\end{proof}

  \vspace{4mm}

\noindent{\bf Acknowledgements.}

\vspace{4mm}
 The  author  thanks  Prof. Rainer Buckdahn for his
careful reading and helpful suggestions. The author also thanks the
editor and  anonymous referees for their helpful suggestions. This
work is supported by the Young Scholar Award for Doctoral Students
of the  Ministry of Education of China and the Marie Curie Initial
Training Network (PITN-GA-2008-213841).

\end{document}